\documentclass[reqno, 11pt]{amsart}
\usepackage{srcltx}
\usepackage[mathscr]{eucal}
\usepackage{amscd}
\usepackage{color}
\usepackage{amsfonts}
\usepackage{cleveref}
\usepackage{amsmath,amssymb,amsfonts,amsthm,bbm,mathrsfs,verbatim}
\usepackage{latexsym}
\usepackage{enumerate}
\usepackage[misc]{ifsym}
\usepackage{fancyhdr}
\usepackage{xr-hyper}

\usepackage{graphics}
\usepackage{epsfig}
\usepackage{xcolor}
\usepackage{graphicx}
\usepackage{epstopdf}

 \newtheorem{thm}{Theorem}[section]

 \newtheorem{defn}[thm]{Definition}
 \newtheorem{rem}[thm]{Remark}
 
 \numberwithin{equation}{section}

\def\R{{\Bbb R}}

\def\la{{\lambda}}
\def\pl{{\partial}}
\def\R{{\mathbb R}}

\def\no{{\nonumber}}

\def\Om{{\Omega}}
\def\bge{\begin{eqnarray}}
\def\bgee{\begin{eqnarray*}}
\def\ege{\end{eqnarray}}
\def\egee{\end{eqnarray*}}

\newcommand{\ee}{{\mathrm{e}}}

\newcommand{\Pbb}{\mathbb{P}}

\def\be{\begin{equation}}
\def\ee{\end{equation}}
\def\bse{\begin{subequations}}
\def\ese{\end{subequations}}
\def\bge{\begin{eqnarray}}
\def\bgee{\begin{eqnarray*}}
\def\ege{\end{eqnarray}}
\def\egee{\end{eqnarray*}}

\date{\today}

\date{\today}

\begin{document}
\title[A stochastic MEMS problem]
{Impacts of noise on quenching of some models arising in MEMS technology}

\author{Ourania Drosinou}
\address{Department of Mathematics, University of Aegean,
Gr-83200 Karlovassi, Samos, Greece
}
\email{rdrosinou@aegean.gr}

\author{Nikos I. Kavallaris}
\address{Department of Mathematical and Physical Sciences, University of Chester, Thornton Science Park,
Pool Lane, Ince, Chester  CH2 4NU, UK}
\email{n.kavallaris@chester.ac.uk}

\author{Christos V. Nikolopoulos}
\address{Department of Mathematics, University of Aegean,
Gr-83200 Karlovassi, Samos, Greece
}

\email{cnikolo@aegean.gr}



\date{\today}


\pagenumbering{arabic}


\begin{abstract}
 In the current work we study a stochastic  parabolic problem.  The underlying  problem is actually motivated by the study of an idealized electrically actuated MEMS  (Micro-Electro-Mechanical System)  device in the case of random fluctuations of the potential difference controlling the device.
 We first  present 
  the mathematical model and then we deduce  some local existence results. Next for some particular versions of the model, regarding its boundary conditions, we derive   quenching  results  as well as estimations of the probability for such singularity to occur. Additional numerical study of the problem in one dimension follows, investigating the problem further with respect to its quenching behaviour.
\end{abstract}

\maketitle
\vspace{0.5in}

\section{Introduction}

In the present work we investigate  the following stochastic semilinear parabolic problem

\bse\label{LSP}
 \be\label{LSP1}
\frac{\partial u}{\partial t} =\Delta u+ \frac{\lambda}{\left(1-u\right)^2 }+\kappa (1-u)\partial_t W(x,t),
 \quad \mbox{in} \quad Q_T:=D \times (0,T),\;T>0,
 \ee
 \be \label{LSP2}
 \mathcal{B}u=\beta_c, \quad \mbox{on} \quad \Gamma_T:=\pl D \times (0,T),
  \ee
  \be\label{LSP3}
 0\leq u(x,0)=u_0(x)<1, \quad x \in D,
   \ee \ese
 as well as some of  its variations rise a  mathematical interest. Here $\lambda$ and $\kappa$,  are given positive constants and  $D$ is  a bounded subset of $\mathbb{R}^d$, $d=1,2,3$ with smooth boundary.
 In addition $\beta_c$ might be a positive  or zero constant whilst the boundary  operator $\mathcal{B}$ gives rise to Robin boundary conditions, i.e. $\mathcal{B}u:=\frac{\partial u}{ \partial \nu} + \beta u  $, for some positive constant $\beta$. 

Remarkably, by setting $\beta\to \infty$ and $\beta_c=0$ (no string effect at the devises support,  cf. \cite{DKN19},  and no external force at it) we obtain Dirichlet boundary conditions. 
 On the other hand, for $0<\beta< \infty$ and  $\beta_c=0$  Robin boundary conditions arise  (models a string effect
in the boundary cf. \cite{DKN19}). Cases for $\beta_c> 0$ can also be considered,  modelling, additional to the string effect, external forces like in pressure sensors,  cf. \cite{Younis}. The latter consideration, regarding  nonhomogeneous boundary conditions,  has  significant theoretical interest as well.
Besides,  the term $\partial_t W(x,t)$ denotes by convention the formal time derivative of the one dimensional real valued Wiener random process $W(x,t)$ in a complete probability space $\{ \Omega,\,{\mathcal F}_t,\,\mathbb{P} \}$ with filtration $\left({\mathcal F}_t\right)_{t\in[0,T]};$ $W(x,t)$ is  defined rigorously in section \ref{pre}. Thus  $\kappa (1-u)\partial_t W(x,t)$ represents a multiplicative noise reflecting the fact of the occurrence of of possible fluctuations into the physical parameters of the MEMS device, cf. section \ref{cmm}.


 Notably towards the limit $\kappa \to 0+$ problem \eqref{LSP} is reduced to its deterministic version, 
\bse\label{DLPu}
 \be\label{DLPu1}
\frac{\partial u}{\partial t}=\Delta u+ \frac{\lambda}{\left(1-u\right)^2 },
 \quad \mbox{in} \quad Q_T,\;T>0,
 \ee
 \be \label{DLPu2}
\mathcal{B}u=0, \quad \mbox{on} \quad \Gamma_T,
  \ee
  \be\label{DLPu3}
 0\leq u(x,0)=u_0(x)<1, \quad x \in D,
   \ee \ese
which, 
 for homogeneous boundary conditions, has been extensively studied in \cite{EGG10, FMPS07, GG08,  KMS08, KS18}. For hyperbolic modifications of  the deterministic variation of  \eqref{LSP} an  interested reader can check \cite{F14, G10,  KLNT15}.  Finally,  non-local alterations of  parabolic and hyperbolic problems arising in MEMS technology are treated in \cite{DKN19, DZ19, GHW09, gkwy20, GK12, KLNT11, KLN16, KS18, M1, M2, M3}.

Due to the presence of the term $f (u):=\frac{1}{(1-u)^2}$ in \eqref{DLPu} we have the occurrence of an singular behaviour, called  {\it (finite-time) quenching}, when $\max_{x\in \bar{D}}u\to 1,$ which is closely associated with the mechanical phenomenon of {\it touching down}. Relating to the stochastic problem \eqref{LSP},  it is worth investigating whether  such a problem can perform analogous singular (quenching) behaviour  to the deterministic problem \eqref{DLPu}.  The main purpose of the current paper is twofold; first to examine the circumstances under which {\it quenching} occurs for the stochastic problem \eqref{LSP}, which is actually a stochastic perturbation of  \eqref{DLPu} derived by a random perturbation of the parameter $\lambda,$ cf. section \ref{cmm}. Secondly,  we intend  to obtain, using both analytical and  numerical approach,  estimates of the {\it probability of quenching} as well as of the {\it quenching time}, which in that case is a random variable.  To the best of our knowledge, this  is the very first time in the literature that this second approach  is considered in the context of MEMS problems. Apart from its practical importance for MEMS engineers such a consideration has its own theoretical value in the context of singular stochastic PDEs (SPDEs).
 
The structure of the current work is as follows. In the next section a derivation of the stochastic model  \eqref{LSP} is presented. In section \ref{pre} we provide the main mathematical tools from stochastic calculus used through the manuscript as well as give the concepts of solutions for the stochastic problem \eqref{LSP} and its considered variations. Section \ref{le} deals with the local existence of all the underlying versions of \eqref{LSP} via Banach's fixed point theorem. Next,  in section \ref{eqp} we appeal to the key properties of exponential functionals of Brownian motion to derive estimates of quenching time as well as estimates of the quenching probability for stochastic problem \eqref{LSP} and some of its variations. As far as we know this  is the first time in the literature of SPDEs where such an approach is used for MEMS nonlinearities. A numerical approach delivered in section  \ref{na} verifies through various numerical experiments the analytical results of the previous sections for nonhomogeneous conditions. Besides, the numerical approach also provides quenching results for the case of homogeneous boundary conditions which is not treated via the analysis of section \ref{eqp}. The current work closes with discussion of the importance of the obtained results in section \ref{di}. 


\section{ The mathematical model}\label{cmm} 
Our main motivation for investigating  problem  \eqref{LSP}  is  its close connection with the operation of
some electrostatic actuated MEMS. By the  term ``MEMS"  we more precisely refer to precision devices
which combine both mechanical processes
with electrical circuits. MEMS devices range in size from
millimetres down to microns, and involve precision mechanical
components which  can be constructed using semiconductor
manufacturing technologies. Indeed, the  last decade various electrostatic actuated MEMS
have been developed and used in a wide variety of devices applied as sensors
and have fluid-mechanical, optical, radio frequency (RF),
data-storage, and biotechnology applications.
Interesting examples of microdevices of this kind include microphones,
temperature sensors, RF switches, resonators, accelerometers,
micromirrors, micropumps, microvalves, data-storage devices etc.,
\cite{KS18, JAP-DHB02,Younis}.

The key part of such a electrostatic actuated
MEMS device usually consists of  an
elastic plate (or membrane) suspended above a rigid ground one.
Regularly the elastic plate is held
fixed at two ends while the other two edges remain free to move,  see Figure~ \ref{Figmems1}.
 \begin{figure}[h!]\vspace{-2cm}
  \centering
\includegraphics[width=.7\textwidth]{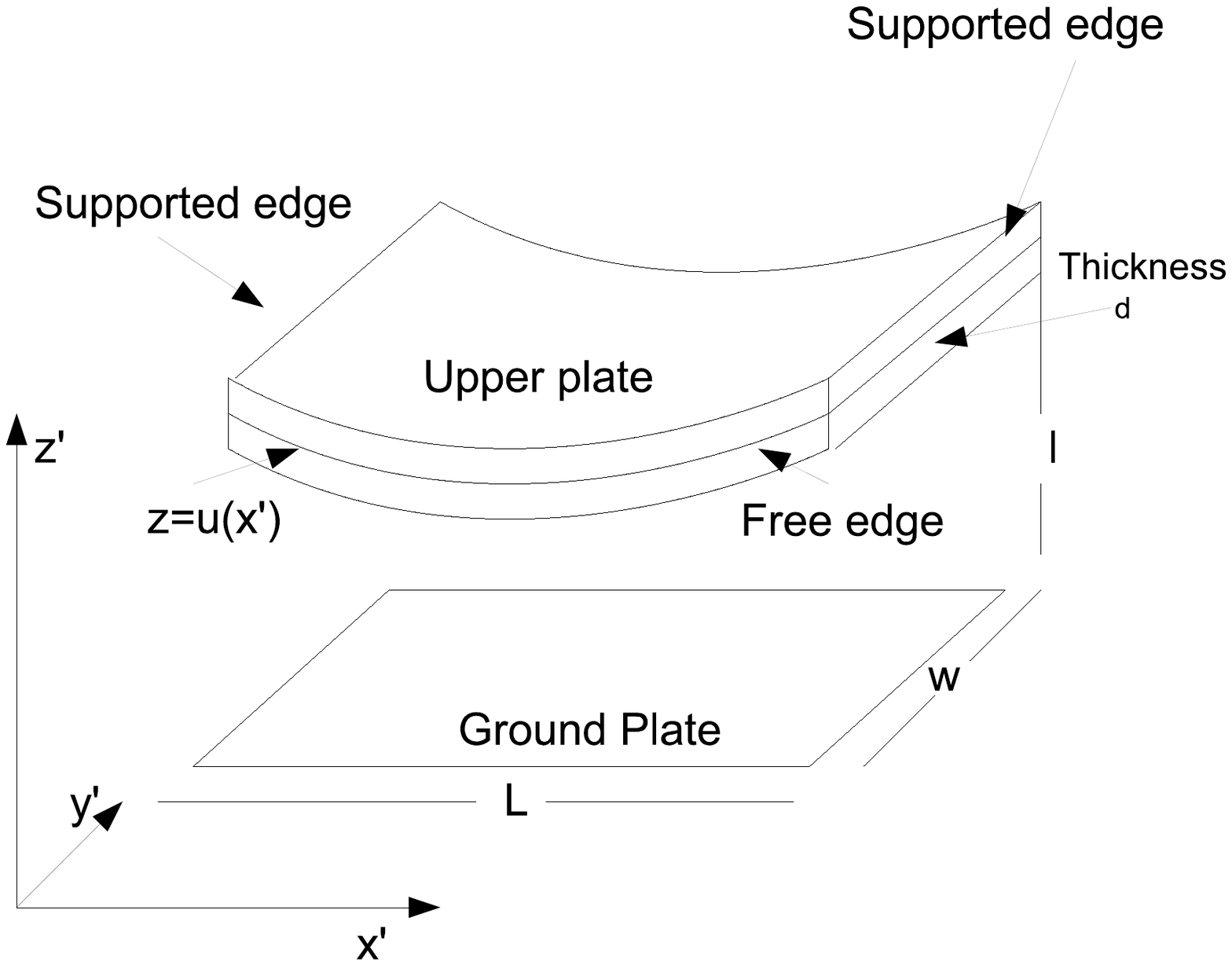}\vspace{-6cm}
  \caption{\it Schematic representation of a MEMS device }
\label{Figmems1}
\end{figure}

When a potential difference
$V$  is applied  between the elastic membrane and the rigid ground plate,
then a deflection of the  membrane towards the plate is observed. Assuming now that  the width $d$ of the gap, between the membrane
and the bottom plate, is small compared to the
device length $L$,  then
the deformation of the elastic membrane $u,$ after proper scaling,  can
described by the dimensionless equation
\bge\label{ik3}
\frac{\partial u}{\partial t}=\Delta u +\frac{\widetilde{\lambda}\, h(x,t)}{(1-u)^2},
\quad x\in D,\; t>0, \label{ik2}
\ege
see \cite{KS18, JAP-DHB02, PT01}. Here the term  $h(x,t)$ describes the varying dielectric properties of the membrane and for some elastic materials can be taken to be constant;
for simplicity henceforth  we  assume  that $h(x,t)\equiv 1,$ although the general case is again considered in section \ref{eqp}.
Besides, the parameter $\lambda$ appears in \eqref{ik3} equals to
$$
\widetilde{\la}=\frac{V^2 L^2 \varepsilon_0}{2\mathcal{T}\ell^3},
$$
and  is actually the {\it tuning} parameter of  the  considered MEMS device.
Note that $\mathcal{T}$ stands for the tension of the elastic membrane,
 $\ell$ is the characteristic width of the gap
between the membrane and the fixed ground plate (electrode),
whilst  $\varepsilon_0$ is the permittivity of free space. MEMS engineers are interested in identifying under which conditions the elastic membrane could touch the rigid plate, a phenomenon is usually called {\it touching down} and could lead to the destruction of MEMS device. Touching down can be described via model \eqref{ik3} and occurs when the deformation $u$  reaches the value $1;$ such a situation  in the mathematical literature is known as {\it quenching (or extinction)}.

Experimental observations, see \cite{Younis},  show a significant uncertainty regarding the values of $V$ and $\mathcal{T}.$  More specifically, $V$ fluctuates around an average value $V_0 $
 (corresponding to some $\la>0$)  inferring that we end up with the parameter $\widetilde{\la}=\la+\sigma\, \eta(x,t)$ where $\sigma>0$ is a coefficient  measuring the  intensity of the fluctuation (noise term)  $\eta(x,t).$ Naturally, the coefficient $\sigma$  depends on the deformation $u$ (that is  $\sigma\equiv \sigma(u)$),  whereas a feasible choice for the  noise $\eta(x,t)$  could be a space-time {\it white noise},  i.e. $\eta(x,t)=\partial_t W(x,t),$ and thus we  consider  $\widetilde{\la}=\la+\sigma(u)\partial_t W(x,t).$  From the applications point of view it would be compelling to investigate the impact of uncertainty on the phenomenon of touching down. Accordingly,  it would be  feasible to choose the diffusion coefficient $\sigma(u)$ as a power of the difference $1-u,$ i.e. $\sigma(u)=(1-u)^{\vartheta},$   measuring the distance to quenching (touching down).
  Now choosing $\theta=3$ we derive
\bgee
\frac{\widetilde{\lambda}}{(1-u)^2}=\frac{\lambda}{(1-u)^2}+\kappa (1-u) \partial_t W(x,t),
\egee
and thus the above analysis reveals that under some imposed uncertainty model \eqref{ik3} can be  transformed to \eqref{LSP1}. Furthermore, it should be noted that the choice $\theta=3$  leads to a linear type diffusion term for which case the local existence theory is well established, cf \cite{dazab, Kavallaris2016, KY20}. For the case of a model with a general  diffusion term $\sigma(u)$ the interested reader can check  \cite{Kavallaris2016}.  Next, in case  the two edges of the membrane are attached to a pair of torsional and translational springs, modeling a flexible non ideal support \cite{DKN19, Younis}, see also  Figure~\ref{Figmems2}, 
 then homogeneous boundary conditions of the form \eqref{LSP2}, with $\beta_c=0$,
 are imposed  together with the stochastic equation for the deformation $u$  and complemented  with  initial condition \eqref{LSP3}.

 The case of having  $\beta_c>0$ may arise as well with a configuration where  the support or cantilever of MEMS devises might be nonideal and flexible. More specifically,   considering the  situation in which together with the spring force at the edges of the membrane we also have a significant 
external force oposite to the spring force,  e.g. due to gravity,  cf. \cite{Younis}. The latter consideration would result in a boundary condition of the form $\frac{\partial u}{\partial n}= -\beta u + \beta_c$ where $\beta_c$ stands for this external force. For simplicity and without loss of generality, especially regarding the analysis in section \ref{eqp}, 
we may take $\beta_c$ to be of the same magnitude as $\beta.$ Then we end up with a nonhomogeneous  boundary condition of  the form 
$\frac{\partial u}{\partial n}= \beta(1- u )$ for some $ \beta>0$.

\begin{figure}[!htb]\vspace{-.4cm}
   \centering
     \includegraphics[width=.6\linewidth]{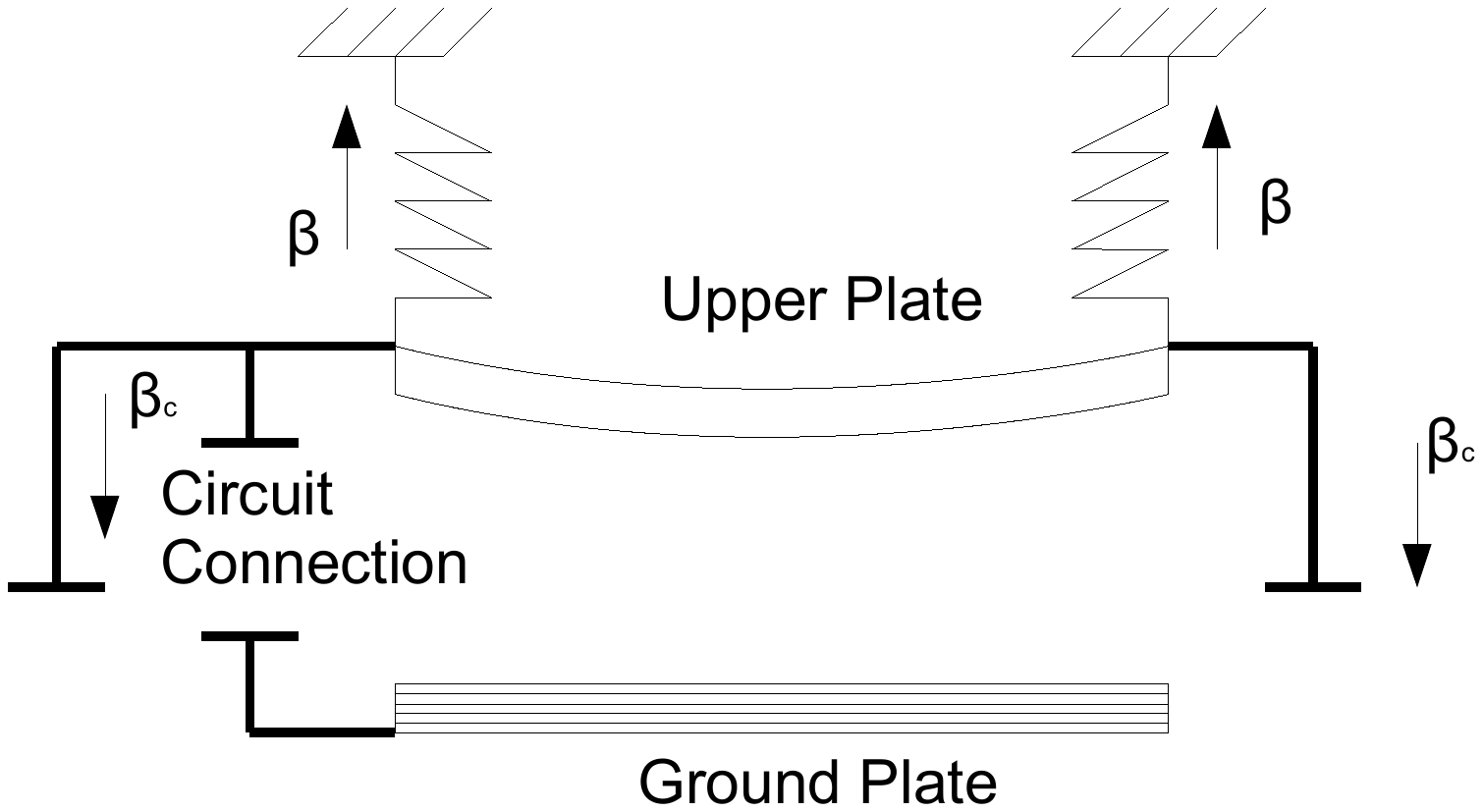}\vspace{-6cm}

     \vspace{-2cm}
\caption{Schematic representation of a MEMS device with support nonideal  and subject to external forces.}\label{Figmems2}
\end{figure}
Notably, the mathematical model \eqref{LSP1}, as a stochastic perturbation of \eqref{ik3}, is build up to capture  possible destructions due to the uncertainty in  parameter measurements of the MEMS system. Thus, under these circumstances is more realistic compared to  \eqref{ik3}.

\section{Preliminaries}\label{pre}
The current section is devoted to  the introduction of the main mathematical concepts and tools  from the area of stochastic calculus that will  be used  throughout the manuscript.  Henceforth, $C,K$ will denote positive constants whose values might change from line to line.

 We  first consider the complete probability space $\{ \Omega,\,{\mathcal F}_t,\,\mathbb{P} \}$ with filtration $\left({\mathcal F}_t\right)_{t\in[0,T]}.$  Next take  $H:=L^2(D)$ and  let also $Q \in  \mathcal{L}_1(H)$ be a linear non-negative definite and symmetric operator which has an orthonormal basis $\chi_{j}(x) \in H, j=1,2, 3, \dots$ of eigenfunctions with corresponding eigenvalues $ \gamma_{j} \geq 0, j=1, 2, 3, \dots$ such that $\text{Tr} (Q) = \sum_{j=1}^{\infty} \gamma_{j} <\infty;$ that is  $Q$ is of trace class.
   Then  $W(\cdot,t)$  is a $Q$-Wiener process if and only if
\begin{equation} \label{noise1}
W(x,t) = \sum_{j=1}^{\infty} \gamma_{j}^{1/2}  \chi_{j}(x) \beta_{j}(t), \quad\mbox{almost surely (a.s.)}\;,
\end{equation}
where $\beta_{j}(t)$ are independent and identically distributed (i.i.d)  $\mathcal{F}_{t}$-Brownian motions and the series converges in $L^{2}(\Om, H),$  cf. \cite{cho}.  It is worth noting that the eigenfunctions $\{ \chi_{j}(x) \}_{j=1}^{\infty}$ may be different from the eigenfunctions $\{ \phi_{j}(x) \}_{j=1}^{\infty}$ of the elliptic operator  $ A=-\Delta:\mathcal{D} ( A)=W^{2,2}(D)\cap W^{1,2}(D)\subset H\to H, $  which is self-adjoint, positive definite with compact inverse.  Note that the trace class operator $Q$ is also a Hilbert-Schmidt operator and then we denote $Q\in \mathcal{L}_2(H).$

For  such an operator  $Q\in \mathcal{L}_2(H)$ with $\text{Tr} (Q) <\infty$,  there exists a kernel $q(x, y)$ such that
\[
(Qu) (x) := \int_{D} q(x, y) u(y) \, dy, \quad\mbox{for any} \; x \in D, \; u \in H,
\]
see \cite[p. 42-43]{cho} and \cite[Definition 1.64]{Lord}.
The kernel $q(x, y)$ is also called the covariance function of the $Q$-Wiener process $W(x,t)$.

Let $X$ be a Banach space with the norm $\| \cdot \|_{X}$ we then define the following Hilbert space
$$
 L_{2}^{0}(H; X) =\left\{\psi \in  L(H,X): \;     \sum_{j=1}^{\infty}   \| \psi Q^{1/2} (\phi_{j})  \|_{X}^{2} = \sum_{j=1}^{\infty} \gamma_{j} \| \psi (\phi_{j})  \|_{X}^{2} <\infty \right\},
$$
with norm $\| \psi \|_{L_{2}^{0}} =  \Big ( \sum_{j=1}^{\infty} \gamma_{j} \| \pi (\phi_{j})  \|_{X}^{2} \Big )^{1/2}$,
 where $L(H, X)$ denotes  the space of all bounded operators from $H$ to $X$.  For $\Psi: [0, T] \to L_{2}^{0} (H, X)$,  the stochastic integral $\int_{0}^{T} \Psi (t) \, d W(t)$ is well defined, \cite{dazab}.
Furthermore we denote by $L^2(\Om,H)$ the space of all random variables $X:\Om\to H$ equipped with the norm $$\|X(\omega)\|_{L^2(\Om,H)}:=\mathbb{E}\left[\|X(\omega)\|_{H}^2\right]^{1/2}
<\infty,\quad\mbox{for any}\quad\omega\in \Om,$$ known also as the space of the mean-square integrable random
variables, where $\mathbb{E}[\cdot]$ stands for the expectation in the probability space $(\Om,\mathcal{F}_t, \mathbb{P}).$

 Analogously  problem  \eqref{LSP} is written in the form of an It\^o problem as follows
\bse\label{LSP1u}
 \be\label{LSPu1}
du_t=\left(\Delta u_t+f(u_t)\right)dt+\sigma( u_t)  dW_t, \quad \mbox{in} \quad Q_T,
 \ee
 \be \label{LSPu2}
0\leq u_0\leq 1, \quad\mbox{almost surely (a.s.)},
   \ee \ese
where  $f(u_t):=(1-u_t)^{-2}$ and $\sigma(u_t):=\kappa(1- u_t).$

  It can be easily checked that  $f:H\to H,$ satisfies a local Lipschitz condition, i.e. for any $0\leq \rho< 1$ and $w_1, w_2\in B_{\rho}:=\{w\in L^{\infty}(\Omega):0\leq\|w\|_{\infty}<\rho\}$  there exists $C_{\rho}>0$ such that
\begin{equation} \label{fnew1}
\|f(w_{1}) - f(w_{2}) \|_{H} \leq C_{\rho}\| w_{1} - w_{2} \|_{H}.
\end{equation}
Notably, an  immediate consequence of \eqref{fnew1} is the following growth condition
\bge\label{gc1}
\|f(w))\|_{H} \leq C_{\rho}\left(1+\| w\|_{H}\right)\quad\mbox{for any}\quad w\in B_{\rho}.
\ege
 Besides,
$\sigma: H\to \mathcal{L}_0^2$  satisfies a local Lipschitz condition and a linear growth condition as well (\cite[Lemma 10.24]{Lord}), in particular for any $0<\rho_1< \rho_2\leq 1$ there exists $K_{\rho_1, \rho_2}>0$ such that for any $w_1,w_2, w\in B_{\rho},$
\begin{equation} \label{fnew1a}
\|\sigma(w_{1}) - \sigma(w_{2}) \|_{\mathcal{L}_0^2} \leq K_{\rho_1, \rho_2} \| w_{1} - w_{2} \|_{H}\quad\mbox{and}\quad \|\sigma(w)\|_{\mathcal{L}_0^2} \leq K_{\rho}\left(1+\| w\|_{H}\right).
\end{equation}
 Then  $u_t=u(\cdot,t)$ can be interpreted as a predictable $H-$valued stochastic process. Next  recalling that  $A=- \Delta :\mathcal{D}(A)=W^{2,2}(\Omega) \cap W^{1,2}(\Omega) \subset H \rightarrow H$ then $-A$  is a generator of an analytic semi group  $\mathcal{G}(t)= e^{-tA}$ on $ H.$

  In the following we  introduce some concepts of solutions for  problem \eqref{LSP1u} that will be used  through the manuscript.
\begin{defn} \label{13}
A predictable $H-$valued stochastic process ${u_t: t \in [0,T]}$ such that
$$\mathbb{P}\left[\sup_{(x,t) \in D\times  [0,T]}|u_t(x)|<1\right]=1,$$
is called a weak solution of problem \eqref{LSP1u}  if for any
$v \in  \mathcal{D}(A)$ and for any $t \in [0,T]$,
\begin{equation} \label{14}
(u_t,v)=(u_0,v)+\int_0^t [-(u_s,Av)+(\lambda f(u_s),v)] ds + \int_0^t (\sigma(u_s) dW_s,v),  \quad  \mathbb{P}-a.s.,
\end{equation}
where $(\cdot,\cdot)$ stands for the inner product into   Hilbert space $ H=L^2(D).$   Note that the stochastic integral $\int_0^t (\sigma(u_s) dW_s,v)$ is well defined, cf. Theorem 2.4 in \cite{cho}.
\end{defn}
\begin{defn} \label{15}
A predictable $ H-$valued stochastic process ${u_t: t \in [0,T]}$ such that $$\mathbb{P}\left[\sup _{(x,t) \in D\times  [0,T]}|u_t(x)|<1\right]=1,$$  is called a mild solution of \eqref{LSP1u}  if for any $t \in [0,T],$  there holds
\begin{equation}
\label{16}
u_t=\mathcal{G}(t) z_0+\lambda \int_0^t \mathcal{G}(t-s) f(u_s)ds +\int_0^t \mathcal{G}(t-s) \sigma(u_s) dW_s, \quad  \mathbb{P}-\mbox{a.s.}\quad\mbox{and   a.e. in} \quad D.
\end{equation}
\end{defn}
Besides, the following interesting variation of problem  \eqref{LSP1u}  is also investigated in the current work
\bse\label{GSPz}
 \be\label{GSPza}
 du_t= \left( g(t) \Delta u_t+\lambda h(x,t) f(u_t)\right)dt+\kappa(t) (1-u_t) dW_t,\quad\mbox{in}\quad Q_T,
 \ee
 \be \label{GSPzb}
0\leq u_0<1,\; a.s.\;,
   \ee \ese
     where  $g, \kappa: \mathbb{R}_+\rightarrow \mathbb{R}_+$  and
   $h: D\times \mathbb{R}_+ \rightarrow \mathbb{R}_+$ are continuous and bounded funtions. It is also assumed that $g\in C^{1}(\mathbb{R}_+).$

Notably,  under the given assumptions for $g,$ cf. \cite{SV03},  then the Green's function $G$ associated with the deterministic problem
\bgee
&&\zeta_t =g(t)\Delta \zeta,
 \quad \mbox{in} \quad Q_T,\\
&& \mathcal{B}(\zeta)=\beta_c ,  \quad \mbox{on} \quad \Gamma_T,\\
 &&  0\leq \zeta(x,0)=\zeta_0(x)<1, \quad x \in D,
\egee
exists and   satisfies the growth conditions
\bge\label{GE}
\left|\partial^m_x\partial^{\ell}_t G(x,t;y,s)\right|\leq c(t-s)^{-\frac{d+|m|+2\ell}{2}}\exp\left[-\frac{|x-y|^2}{t-s}\right],
\ege
where  $m= (m_1, . . .,m_d)\in \mathbb{N}^N, \ell \in \mathbb{N}$  and $|m|+2\ell\leq 2,\; |m|=\sum_{j=1}^N m_j.$

Then we define the corresponding semigroup $\mathcal{E}(t)$ on $H=L^2(D)$  as follows
\bge\label{sdn}
\mathcal{E}(t)w(x):=\int_D G(x,t;y,0)\,w(y)dy\quad\mbox{for any}\quad x\in D\quad\mbox{and}\quad 0<t<T,
\ege
 for any  $T>0.$
 
Using estimates \eqref{GE} in conjunction with a standard approach, cf. \cite[Lemma 5.1]{cho} we then obtain the following key estimate
\begin{equation} \label{23}
\int_0^t \left \Vert \mathcal{E}(s) \right \Vert_{\mathcal{L}(H)} ^2 ds\leq K_1^2T, \quad 0<t<T.
\end{equation}
In a similar manner we define the notion of weak and mild solutions for problem \eqref{GSPz}.
\begin{defn} \label{13a}
A predictable $ H-$valued stochastic process ${u_t: t \in [0,T]}$ such that
$$\mathbb{P}\left[\sup _{(x,t) \in D\times  [0,T]}|u_t(x)|<1\right]=1,$$
is called a weak solution of problem \eqref{GSPz}  if for any
$v \in \mathcal{D}(A)$ and for any $t \in [0,T]$,
\begin{equation} \label{14a}
(u_t,v)=(u_0,v)+\int_0^t [-g(s)(u_s,Av) + \lambda( h(\cdot,s) f(u_s),v)] ds +\int_0^t ( \kappa(t) (1- u_s) dW_s,v),  \quad  \mathbb{P}-a.s.\,.
\end{equation}
\begin{defn} \label{15a}
A predictable $ H-$valued stochastic process ${u_t: t \in [0,T]}$ such that $$\mathbb{P}\left[\sup _{(x,t) \in D\times  [0,T]}|u_t(x)|<1\right]=1,$$  is called a mild solution of \eqref{GSPz}  if for any $t \in [0,T],$  there holds
\begin{equation}
\label{16a}
u_t=\mathcal{E}(t) u_0 +\lambda \int_0^t \mathcal{E}(t-s) h(\cdot,s)f(u_s)ds +\int_0^t \mathcal{E}(t-s) \widetilde{\sigma}(u_s) dW_s, \quad  \mathbb{P}- \mbox{a.s.}\quad\mbox{ and  a.e. in} \quad D,
\end{equation}
where  $ \widetilde{\sigma}(u_t)=\kappa(t)(1-u_t)$  satisfies clearly condition \eqref{fnew1a} for $\kappa(t)$ bounded.
\end{defn}
\end{defn}

\begin{rem}
 Note that  any  weak (variational) solution is a mild solution under the assumption of the local Lipschitz continuity of $f,$ see \cite{GR00}. Conversely, any regular enough mild  solution is also a weak solution, cf. \cite{Lord}. The weak formulations \eqref{14} and \eqref{14a} will be used in section \ref{eqp} for the investigation of the quenching behaviour, whilst in the following section some existence and uniqueness results for mild solutions are presented.
\end{rem}
Next we recall that It\^o's formula (see \cite[Theorem 5.2 page 88]{1}) entails
\bge\label{rcn1}
F(W_t)-F(W_0)=\int_0^t F'(W_s)dW_s + \frac{1}{2} \int_0^t F''(W_s)ds,
\ege
for any function $F\in C^2(\R),$ which in differential form gives
$$dF(W_t)=F'(W_t) dW_t+ \frac{1}{2} F''(W_t) dt.$$

Closing the current  section we recall the integration by parts  formula for stochastic processes. Indeed, if  $X_t$ and $Y_t$ are It\^o stochastic processes given by
\[
X_t=X_0+\int_0^t\Psi_s\,ds+\int_0^t \Phi_s \, dW_s\quad\mbox{and}\quad  Y_t=Y_0+\int_0^t\tilde{\Psi}_s\,ds+\int_0^t \tilde{\Phi}_s \, dW_s
\]
then
\bge\label{mmk5}
X_tY_t=X_0Y_0 + \int_0^t X_s dY_s +\int_0^t Y_sdX_s+\left[ X,Y\right]_t, \quad t\in[0,T]
\ege
where the last term in the above formula is the quadratic variation  of $X_t, Y_t$  and is defined as
\bge\label{mmk6}
\left[ X,Y\right]_t:=\int_0^t \Phi_s \tilde{\Phi}_s\,ds,
\ege
cf. \cite[Corollary 7.11  page 119]{1}.
\section{Local Existence}\label{le}
In the current we present  local existence and uniqueness results for problems \eqref{LSP1u} and \eqref{GSPz}.

Due to conditions \eqref{fnew1}, \eqref{gc1} and \eqref{fnew1a} we derive the following local-in-time existence and uniqueness  result for problem  \eqref{LSP1u}. 

\begin{thm} \label{thm1}
 Fix \,$0<\rho_0<1$ and consider initial data $u_0 \in L^2(\Om,\mathcal{D}(A))$ such that
 $\parallel u_0 \parallel_{L^2(\Om,\mathcal{D}(A))}<\rho_0$,
  then there exists $T=T(\rho_0)>0$ such that problems \eqref{LSP1u} admits a unique mild solution $u_t$ in $ [0,T]$.
Furthermore, there exists $C_T>0$ such that
\begin{equation}\label{17}
\sup _{0\leq t \leq T} \parallel u_t\parallel_{L^2(\Om,\mathcal{D}(A))}\leq C_T(1+\parallel u_0 \parallel_{L^2(\Om,\mathcal{D}(A))}).
\end{equation}
\end{thm}
\begin{proof}
The proof is based on Banach's fixed point theorem  and it follows along the same lines with the proof of \cite[Theorem 4.4]{Kavallaris2016}  and thus it is omitted.
\end{proof}
 
Next  the existence and uniqueness for solutions of problem \eqref{GSPz} can be directly derived by \cite{SV03}, however in the sequel and for the sake of completeness we provide and prove such a result.
\begin{thm} \label{thm1a}
Suppose that the functions $g(t), \kappa_1(t)$ are bounded  in $[0,T]$, and $h(x,t)$ bounded in $D\times [0,T]$ for any $T>0.$ Assume further that $g(t)\in C^1([0,T]).$
Then for fixed $0<\rho_0<1$ and  initial data
$u_0 \in L^2(\Om,\textit{D}(A))$ such that $\parallel u_0 \parallel_{L^2(\Om,\textit{D}(A))}<\rho_0$,  there exists $T=T(\rho_0)>0$ such that problem \eqref{GSPz} admits a unique mild solution
$u_t$ in $ [0,T].$
Furthermore, there exists $C_T>0$ such that
\begin{equation}\label{eq16}
\sup _{0\leq t \leq T} \parallel u_t\parallel_{L^2(\Om,\textit{D}(A))}\leq C_T(1+\parallel u_0 \parallel_{L^2(\Om,\textit{D}(A))}).
\end{equation}
\end{thm}
\begin{proof}
We first note that since $g(t)$ is bounded then $\mathcal{D}\left(g(t)A\right)=\mathcal{D}\left(A\right)=W^{2,2}(\Omega) \cap W^{1,2}(\Omega)$ recalling that $A=-\Delta.$

Denote by $S_T$ the Banach  space of $H-$valued predictable process ${u_t:t \in [0,T]}$ equipped with the norm
$$\parallel u_t \parallel_{S_T}:= \sup_{0\leq t\leq T} \parallel u_t\parallel _{L^2(\Om,\mathcal{D}(A))}.$$
Now, for any $\rho$, with $0<r_0<\rho<1$ we set,
$$S_{\rho,T}:= \left\lbrace u_t \in S_T: \parallel u_t \parallel _{S_{\rho,T}}:= \sup _{0\leq t \leq T} \parallel u_t \parallel _{L^2(\Om,\mathcal{D}(A))} \leq \rho <1 \right\rbrace$$
and we define the operator  $\mathcal{M}:  S_{\rho,T}\rightarrow S_T$ as follows
\begin{equation} \label{18}
\mathcal{M}(u_t):=\mathcal{E}(t) u_0 +\lambda\int_0^t \mathcal{E}(t-s) f(u_s)ds+\int_0^t \mathcal{E}(t-s)\kappa(t)  \sigma(u_s) dW_s.
\end{equation}
The main idea is to apply  Banach's fixed point theorem  to  prove existence and uniqueness of the equation $\mathcal{M}(u_t)=u_t$ in $S_{\rho,T}.$
\\
$\textbf{\emph{Step 1:}}$  We first show that  $\mathcal{M}$ maps  $S_{\rho,T}$ into itself.
To this end  note that $\mathcal{M}(u_t)$ is a $H-$valued predictable process because $ u_0$ is 
$F_0-$measurable and the stochastic integral is a predictable process. So, it suffices to prove that
 $\parallel \mathcal{M}(u_t)\parallel_{S_{\rho,T}}< \rho$.

By the hypothesis on the initial data $u_0$, and using \eqref{sdn} in conjunction  with \cite[Lemma 5.1]{cho} we have regarding the first term in \eqref{18}
\begin{equation}\label{19}
\parallel \mathcal{E}(t) u_0 \parallel _{S_{\rho,T}}\leq \parallel u_0 \parallel_ {S_{\rho,T}}<\rho_0.
\end{equation}

 Also for $u_t \in S_{\rho,T}$ and by virtue of Sobolev's inequality there is $\rho_0$ small enough such that
\begin{equation}
\label{20}
\mathbb{E}[\parallel u_t \parallel _{\infty}]\leq \rho_0 <1
\end{equation}
 implying that  conditions \eqref{fnew1}-\eqref{fnew1a} hold true.

So using now the growth condition \eqref{gc1}, \eqref{GE}  and  \cite[Lemma 5.1]{cho} we  derive that
\begin{eqnarray} \nonumber
\left\Vert\int_0^t \mathcal{E}(t-s)h(\cdot, s) f(u_s)ds \right\Vert_{S_{\rho,T}}
&\leq &
\int_0^t \left\Vert \mathcal{E}(t-s) h(\cdot, s)f(u_s) \right\Vert_{S_{\rho,T}}ds\leq
N_T\int_0^t \left\Vert f(u_s)\right\Vert_{S_{\rho,T}}ds\\
&\leq &
N_T\int_0^t C_{\rho_0} (1+ \left\Vert u_s \right\Vert_{S_{\rho,T}})ds
\nonumber
\\
&\leq & N_T C_{\rho_0} T \left(1+ \sup_{0\leq s\leq T} \parallel u_s \parallel _{S_{\rho,T}}\right)\leq 2 C_{\rho_0, T} T,\label{18-2nd}
\end{eqnarray}
for any $0<t<T$ where $N_T:=\max_{\bar{D}\times [0,T]} h(x,t)$ and $C_{\rho_0, T}:=N_T C_{\rho_0}.$
Next via  It\^o's isometry, \cite[page 322 ]{Lord},
and  by virtue of   \eqref{fnew1a} and \cite[Exercise 10.7 page 480]{Lord}  we have that for any $0<t<T$
\begin{eqnarray} \nonumber
\left\Vert \int_0^t \mathcal{E}(t-s)\kappa(t) \sigma( u_s) dW_s  \right\Vert _{S_{\rho,T}} ^2
&=&\int_0^t \mathbb{E} \left[ \left \Vert\mathcal{E}(t-s)\kappa(t) \sigma(u_s) \right\Vert _{L_0 ^2} ^2 \right] ds \\
&\leq & \left(\kappa_T  K_{\rho_0}\right) ^2 \int_0^t \left\Vert  \mathcal{E} (t-s) \right\Vert _{\mathcal{L}(H)} ^2 ds \left( 1+ \sup_{0\leq s \leq t} \left\Vert u_s \right \Vert _{S_{\rho,T}} \right)^2 \qquad\label{22}
\end{eqnarray}
where $\kappa_T:=\max_{[0,T]}\kappa(t).$

Recalling now the  semigroup estimate  \eqref{23} then \eqref{22} finally reads
\begin{equation} \label{24}
\left\Vert \int_0^t \mathcal{E}(t-s)\kappa(t)  \sigma(u_s) dW_s  \right\Vert _{S_{\rho,T}} \leq 2 \widetilde{K}_TT^{\frac{1}{2}},\quad 0\leq t \leq T
\end{equation}
for $\widetilde{K}_T$ a positive constant.

Finally, combining  the above relations,  \eqref{19},\eqref{18-2nd} and \eqref{24}
 we derive from equation \eqref{18} that
\begin{equation}\label{25}
\left \Vert \mathcal{M}(u_t) \right \Vert_{S_{\rho,T}}<\rho_0+2 \lambda C_{\rho_0, T}T + 2 \widetilde{K}_T T^{\frac{1}{2}}\leq \rho_0+2 \left(\lambda C_{\rho_0}T +  \widetilde{K}_T T^{\frac{1}{2}}\right).
\end{equation}
Next taking  $T$  small enough, say smaller than some $T_1=(\lambda,\rho, \rho_0),$ such that
$$2 \left(\lambda C_{\rho_0, T}T +  \widetilde{K} T^{\frac{1}{2}}\right)< \rho-\rho_0, \quad \mbox{for any} \quad 0<T<T_1,$$
then  (\ref{25}) reads
\begin{equation}\label{26}
\left \Vert \mathcal{M}(u_t) \right \Vert_{S_{\rho,T}}\leq \rho_0+\rho-\rho_0=\rho,
\end{equation}
hence  $\mathcal{M}$ maps  $S_{\rho,T}$ to itself.

$\textbf{\emph{Step 2:}}$ Next we show that  $\mathcal{M}$ is a contraction operator, that is  there is a positive constant $0<\gamma<1$ such that
\begin{equation}\nonumber
\left\Vert \mathcal{M}(u_t)-\mathcal{M}(v_t)\right\Vert_{S_{\rho,T}}\leq \gamma \left\Vert u_t-v_t \right\Vert_{S_{\rho,T}}.
\end{equation}
In fact 
\begin{eqnarray}
\mathcal{M}(u_t)-\mathcal{M}(v_t)
 &=&\lambda
 \int_0^t \mathcal{E}(t-s) h(\cdot,s) \left(f(u_s)-f(v_s)\right) ds  +\int_0^t \mathcal{E}(t-s) \kappa(s) \left( v_s- u_s \right)dW_s,
 \nonumber
\end{eqnarray}
which implies
\begin{eqnarray}
\left\Vert \mathcal{M}(u_t)-\mathcal{M}(v_t)\right\Vert_{{L^2(\Om,\mathcal{D}(A))}} ^2 &=&
\left\Vert \lambda\int_0^t \mathcal{E}(t-s) h(\cdot,s)\left(f(u_s)-f(v_s)\right) ds\right.\nonumber\\
 &+& \left.  \int_0^t \mathcal{E}(t-s) \kappa(s) \left(  v_s-u_s \right)dW_s\right\Vert_{L^2(\Om,\mathcal{D}(A))}^2
\nonumber\\
&\leq &
(\lambda N_T)^2 \left\Vert \int_0^t\mathcal{E} (t-s) \left(f(u_s)-f(v_s)\right) ds\right\Vert_{L^2(\Om,\mathcal{D}(A))}^2\nonumber\\
&+& \kappa^2_T\left\Vert\int_0^t \mathcal{E}(t-s) \left(  v_s- z_s \right)dW_s\right\Vert_{L^2(\Om,\mathcal{D}(A))}^2\label{ub1}.
\end{eqnarray}
The first term in the RHS of  \eqref{ub1} using \eqref{fnew1}, \eqref{GE} and  \cite[Lemma 5.1]{cho}  is estimated as  follows
\bge\label{ub2}
(\lambda N_T)^2 \left\Vert \int_0^t\mathcal{E} (t-s) \left(f(u_s)-f(v_s)\right) ds\right\Vert_{L^2(\Om,\mathcal{D}(A))}^2\leq(\lambda N_T C_{\rho_0} T)^2  \left \Vert u_t-v_t \right \Vert_{S_{\rho,T}} ^2 .
\ege
On the other hand, relation \eqref{fnew1a}  in conjunction  with  It\^o's  isometry,  \eqref{23}  and  \cite[Exercise 10.7 page 480]{Lord} infers
\bge\label{ub3}
\kappa^2_T\left\Vert\int_0^t \mathcal{E}(t-s) \left( u_s- v_s \right)dW_s\right\Vert_{L^2(\Om,\mathcal{D}(A))}^2\leq \kappa^2_T K_1^2 T  \left \Vert u_t-v_t \right \Vert_{S_{\rho,T}} ^2.
\ege
Now \eqref{ub1} by virtue of \eqref{ub2} and \eqref{ub3} reads
\begin{eqnarray*}
\left\Vert \mathcal{M}(u_t)-\mathcal{M}(v_t)\right\Vert_{S{\rho,T}} ^2
 &\leq &
 (\lambda N_T C_{\rho_0} T)^2  \left \Vert u_t-v_t \right \Vert_{S_{\rho,T}} ^2+
  \kappa^2_T K_1^2 T  \left \Vert u_t-v_t \right \Vert_{S_{\rho,T}} ^2
\nonumber\\
&\leq& (\lambda^2 N^2_T C^2_{\rho_0} T^2+  \kappa^2_T K_1^2 T ) \left \Vert u_t-v_t \right \Vert_{S_{\rho,T}} ^2
\nonumber\\
&\leq & (\lambda^2 N^2_T C^2_{\rho_0} +  \kappa^2_T K_1^2 ) T \left \Vert u_t-v_t \right \Vert_{S_{\rho,T}} ^2\nonumber\\
&\leq &  \frac{1}{2} \left \Vert u_t-v_t \right \Vert_{S_{\rho,T}} ^2
\end{eqnarray*}
provided that $$T< T_2:=\min\left\{1, \frac{1}{2(\lambda^2 N^2_T C^2_{\rho_0} +  \kappa^2_T K_1^2 )}\right\},$$
and  thus  $\mathcal{M}$  is a contraction in $S_{\rho,T}$ provided that  $T<T_2.$

Consequently, by choosing $T_0=\min\{T_1,T_2\},$ we derive that $\mathcal{M}$  has a unique fixed point in $S_{\rho,T}$ for $0<T<T_0$ by Banach's fixed point theorem and thus problem \eqref{GSPz}  has a unique mild solution in the time interval  $[0,T_0].$

Accordingly,  a direct application of Gronwall's inequality  infers estimate \eqref{17}.
\end{proof}

\begin{rem}
Due to the obtained regularity, see \eqref{17}, the mild solution provided by Theorem \ref{thm1} is actually a weak solution, cf. \cite{SV03}.
\end{rem}

  Note that if set $z=1-u$  where $u$ is the solution of  \eqref{LSP}  then $z$ satisfies
\bse\label{LSPza}
 \be\label{LSP1za}
\frac{\partial z}{\partial t}=\Delta z-\frac{\lambda}{z^2 }-\kappa z\partial_t W(x,t),
 \quad \mbox{in} \quad Q_T,
 \ee
 \be \label{LSP2za}
  \mathcal{B}(1-z)=\beta_c\quad \mbox{on} \quad \Gamma_T,
  \ee
  \be\label{LSP3za}
 0<z_0(x):=z(x,0)=1-u_0(x)=\xi(x)\leq 1, \quad x \in D.
   \ee \ese
   In particular,  if $u=0$ on $\Gamma_T$ this results in $z=1$  for condition \eqref{LSP2za}, or otherwise into $\frac{\partial z}{ \partial \nu} +\beta  z=0$ if $u$ satisfies the boundary condition $\frac{\partial u}{ \partial \nu} =\beta  (1-u)$.

Accordingly  problem \eqref{LSPza} can be considered as an It\^o equation in the Hilbert space $H=L^2(D)$ and so it can be written by suppressing the dependence on space as follows:
\bse\label{LSP1z}
 \be\label{LSPz1}
dz_t=\left(\Delta z_t-\frac{\lambda}{z^2_t}\right)dt -\kappa z_t  dW_t, \quad \mbox{in} \quad Q_T,
 \ee
 \be \label{LSPz2}0<z_0=\xi\leq1,\; a.s.\;
   \ee \ese
   and its local existence and uniqueness stems from Theorem \ref{thm1}.
   
Besides, if $u_t$ satisfies \eqref{GSPz}  then $z_t=1-u_t$  solves the following  It\^o's problem 
\bse\label{GSPzu}
 \be\label{GSPzua}
 dz_t= \left( g(t) \Delta z_t-\lambda h(x,t) z_t^{-2}\right)dt-\kappa(t)  z_t dW_t,\quad\mbox{in}\quad Q_T,
 \ee
 \be \label{GSPzub}
0<z_0=\xi\leq1,\; a.s.\;
   \ee 
\ese
for which local existence and uniqueness is guaranteed by Theorem \ref{thm1a}.

Remarkably, problems \eqref{LSP1z} and \eqref{GSPzu} are more appropriate for the analysis of the quenching behaviour delivered  in the following section.

\section{ Estimation of Quenching Probability 
}\label{eqp}

\subsection{The basic model \eqref{LSP1z}}\label{bp}
 In the sequel we will first  investigate the quenching behaviour of problem \eqref{LSP1z}, whose solution can be expressed as an  It{\^o} process as follows
\bge\label{mmk1}
z_t=z_0-\kappa \int_0^t z_s dW_s+\int_0^t \left(\Delta z_s-\frac{\la}{z_s^2}\right)\,ds.
\ege
Remarkably, the analysis that follows applies to the  imposed homogeneous Robin boundary condition
$\frac{\partial z_t}{\partial \nu}+\beta z_t =0$ which corresponds to the situation that a boundary condition
\eqref{LSP2} is applied for $\beta_c>0$. The nonhomogeneous Robin boundary condition, arises for  $\beta_c=0$ is treated only numerically in section \ref{na}.

We  define now  the  stochastic process
\be\label{mmk2}
 v_t= e ^{\kappa W_t} z_t, \quad 0\leq t<\tau ,
\ee
cf.\cite{DLM}, where $\tau$  identifies a (random) stopping time, which is actually the quenching time for  both $z_t$ and $v_t. $  In particular, for any stochastic process satisfying  \eqref{mmk1}  there holds
\[
\limsup_{t\rightarrow \tau} \inf_{x\in D}|z_t(x)|= 0, \quad a.s. \quad \mbox{in}
\quad {\tau< +\infty}.
\]

 Next using It\^o's formula \eqref{rcn1} for $F(u)=e^{\kappa u}$  we obtain
 \begin{eqnarray}
 \label{1}
 e^{\kappa W_t} &=& e^{\kappa W_0}  +\kappa \int_0^t  e^{\kappa W_s}dW_s + \frac{\kappa ^2}{2} \int_0 ^t e^{\kappa W_s}ds
 \nonumber\\
 &=&
 1 +\kappa \int_0^t  e^{\kappa W_s}dW_s + \frac{\kappa ^2}{2} \int_0 ^t e^{\kappa W_s}ds,
 \end{eqnarray}
since  $W_0=0,$ or equivalently
\begin{equation}
\label{2}
d(e^{\kappa W_t})=\kappa e^{\kappa W_t}dW_t+\frac{\kappa ^2}{2} e^{\kappa W_t}.
\end{equation}
 In the sequel, we use for simplicity  the notation
\bgee
z_t(\phi):=\int_D z_t \phi\,dx, \; t\geq 0,
\egee
for any function $\phi \in C^2(D).$ 

Then problem (\ref{mmk1}), using also second Green's formula,  can be written in a weak formulation as follows
\begin{eqnarray}
z_t(\phi)= z_0(\phi)
&+& \int_0^t \int_{\partial D} \left[ \frac{\partial z_s}{\partial \nu} \phi
- z_s \frac{\partial \phi}{\partial \nu}\right ] d \sigma ds
+ \int_0^t  z_s(\Delta \phi)ds\nonumber\\
&-& \lambda \int_0^t z_s^{-2}(\phi) ds
-\kappa  \int_0^t z_s(\phi) dW_s,\label{3}
\end{eqnarray}
for some test  function $\phi$ smooth enough, where
\bgee
z^{-2}_s(\phi):=\int_D z^{-2}_s \phi\,dx.
\egee

Next we take as a test function
 $\phi \in C^2(D)$  satisfying 
\begin{eqnarray}
&&-\Delta \phi =\lambda_1 \phi, \quad  x \in D,
\label{eigena}\\
&&\frac{\partial \phi}{\partial \nu}+\beta \phi =0,  \quad x \in \partial D,\label{eigenb}
\end{eqnarray}
normalized as 
\bge 
\int_{D} \phi(x)dx =1\label{eigenc}.
\ege
Note that   the principal eigenvalue $\lambda_1$  is positive for $\beta \neq 0,$  cf. \cite[Theorem 4.3]{Am76}.

 In particular the boundary integral  in \eqref{3} thanks to the applied homogeneous Robin-type boundary conditions gives
$$\int_{\partial D} \left[\frac{\partial z_t}{\partial \nu} \phi -z_t \frac{\partial \phi}{\partial \nu}\right] d\sigma=\int_{\partial D }\left(-\beta z_t \phi + \beta z_t \phi\right) d \sigma=0,$$
and thus the weak formulation  \eqref{3}  reduces to
\bge\label{wfc}
z_t(\phi) =  z_0(\phi)
 + \int_0^t  z_s(\Delta \phi)\, ds
 - \lambda \int_0^t z_s^{-2}(\phi) ds
-\kappa  \int_0^t z_s(\phi) dW_s.
 \ege


Applying now the integration by parts formula \eqref{mmk5}  to the  It\^o's processes defined by \eqref{mmk1} and \eqref{1} we have
\[
v_t=e^{\kappa W_t} z_t=e^{\kappa W_0} z_0+ \int_0^t e^{\kappa W_s} dz_s
+ \int_0^t z_s de^{\kappa W_s}+\left[e^{\kappa W_s}, z_s \right](t), \]
where the quadratic variation is given by
\bge\label{mtk1a}
\left[e^{\kappa W_s}, z_s \right](t) =-\kappa ^2\int_0^t e^{\kappa W_s}z_s\,ds, \quad t\geq 0,
\ege
and thus
\begin{eqnarray}\label{veq1}
v_t=z_0+ \int_0^t e^{\kappa W_s}dz_s+\int_0^t z_s de^{\kappa W_s} - \kappa ^2 \int_0^t e^{\kappa W_s}z_s\,ds.
\end{eqnarray}

Next  multiplying \eqref{veq1} by $\phi$ and integrating over the domain $D$ we obtain
\begin{eqnarray}\label{mmk3}
v_t(\phi) &=& z_0(\phi)
+ \int_0^t  e^{\kappa W_s}\left[\int_D \left(\Delta z_s-\lambda z_s^{-2} \right) \phi\, dx  \right]\,ds
- \kappa  \int_0^t e^{\kappa W_s}z_s(\phi)\, dW_s
\nonumber\\
&&+ \kappa  \int_0^t e^{\kappa W_s}z_s(\phi)\, dW_s
+ \frac{\kappa ^2}{2} \int_0^t e^{\kappa W_s}z_s(\phi)\,ds
 -\kappa ^2 \int_0^t e^{\kappa W_s} z_s(\phi)\,ds
\nonumber\\
&=&z_0(\phi) + \int_0^t  e^{\kappa W_s}\left[\int_D \left(\Delta z_s-\lambda z_s^{-2} \right) \phi\, dx  \right]\,ds
 - \frac{\kappa ^2}{2} \int_0^t z_s(\phi) e^{\kappa W_s}ds
 \nonumber\\
 &=& z_0(\phi) + \int_0^t e^{\kappa W_s}   z_s(\Delta \phi)\, ds
 - \lambda\int_0^t e^{\kappa W_s}  z_s^{-2}( \phi)\, ds - \frac{\kappa ^2}{2} \int_0^t  z_s(\phi) e^{\kappa W_s}\, ds \nonumber\\
&=& z_0(\phi)- \lambda_1 \int_0^t e^{\kappa W_s} z_s(\phi)\, ds  - \lambda\int_0^t e^{\kappa W_s}  z_s^{-2}( \phi)\, ds - \frac{\kappa ^2}{2} \int_0^t  e^{\kappa W_s} z_s(\phi)\, ds ,
\end{eqnarray}
using  also \eqref{LSPz1}  and \eqref{2}  together with second Green's identity.

 Next expressing \eqref{mmk3}   in terms of the  $v_t,$ and  since $z_t=v_t e^{-\kappa W_t},$ then thanks to \eqref{mmk2} we infer
\begin{eqnarray}
v_t(\phi) &=& z_0(\phi)- \lambda_1  \int_0^tv_s(\phi)\, ds
-\lambda \int_0^t e^{3\kappa W_s} v_s^{-2}(\phi)\,ds  - \frac{\kappa ^2}{2} \int_0^t v_s(\phi)\,ds\nonumber\\
&=& v_0(\phi)- \left(\lambda_1+\frac{\kappa^2}{2}\right)\int_0^t v_s(\phi)\,ds-\lambda \int_0^t e^{3\kappa W_s} v_s^{-2}(\phi)\,ds,\quad
\label{veq2}
\end{eqnarray}
since $z_0(\phi)=v_0(\phi)$ due to \eqref{mmk2}.

Then \eqref{veq2} implies
\begin{eqnarray}
\frac{v_{t+\epsilon}(\phi)- v_t(\phi)}{\epsilon} &=&
 \frac{1}{\epsilon} \left[ - \left(\lambda_1+\frac{\kappa^2}{2}\right)\int_t^{t+\epsilon} v_s(\phi)\,ds
-\lambda \int_t^{t +\epsilon} e^{3\kappa W_s} v_s^{-2}(\phi)\,ds\right],
   \label{66}
\end{eqnarray}
and   letting $\epsilon \rightarrow 0 $ in equation (\ref{66}) we derive
 \bge
 \label{veq8}
 \frac{d v_t(\phi)}{dt}=- \left(\lambda_1+\frac{\kappa^2}{2}\right) v_t(\phi)-\lambda e^{3\kappa W_t}  v_t^{-2}( \phi)\,\;t>0,\quad v_0(\phi)>0.
 \ege
By virtue of Jensen's inequality, since $r(s)=s^{-2}, s>0$ is convex,  and  via \eqref{eigenc} we have
\bgee
v_t^{-2}(\phi)=\int_D v_t^{-2} \phi\, dx \geq \left( \int_D v_t \phi\, dx \right)^{-2}=(v_t(\phi))^{-2}
\egee
and thus \eqref{veq8} leads to the  following differential inequality
\bgee
 \frac{d v_t(\phi)}{dt} \leq - \left(\lambda_1+\frac{\kappa^2}{2}\right) v_t(\phi) - \lambda e^{3\kappa W_t}(v_t(\phi))^{-2} , \quad v_0(\phi)>0.
\egee
 By a standard comparison principle we have that $v_t(\phi)\leq B(t)$ where $B(t)$ satisfies the following
 Bernoulli differential equation:
\bgee
B'(t)=- \left(\lambda_1+\frac{\kappa^2}{2}\right) B(t) -\lambda  e^{3\kappa W_t} B^{-2}(t), \quad B_0=B(0)= v_0(\phi)>0,
\egee
and is given by
 \bge\label{mmk4}
 B(t)&=&e^{- \left(\lambda_1+\frac{\kappa^2}{2}\right) t}
 \left[B_0^3-3\lambda \int_0^t e^{3\left[\left(\la_1+\frac{\kappa^2}{2}\right)s+\kappa W_s\right]} ds \right]^{1/3}.
 \ege
 Next taking into account \eqref{mmk4} we can define the stopping (quenching) time for $B(t)$ as
 \bgee
 \tau_1:=\inf \left\lbrace  t\geq 0 \Big| \int_0^t e^{3\left[\left(\la_1+\frac{\kappa^2}{2}\right)s +\kappa W_s\right]} ds \geq \frac{1}{3\lambda} B_0^3 \right\rbrace,
 \egee
 and so it follows that $B(t)$ extincts to zero in finite time on the event $\left\lbrace \tau < + \infty \right\rbrace$. The fact  that $0\leq v_t(\phi)\leq B(t)$ implies that $\tau_1$ is an upper bound of the stopping (quenching) time $\tau$ for $v_t(\phi),$ hence the function
\[
t\mapsto\int_D e^{\kappa W_t} z_t(x)\phi(x)\,dx
\]
quenches in finite time under the event  $\left\lbrace \tau_1 < + \infty\right\rbrace.$ Using now \eqref{eigenc} as well as the fact that $t\mapsto e^{\kappa W_t}$  is bounded away from zero on $[0,\tau_1],$ since $\tau_1$  is finite (cf. \eqref{psk4} and  \eqref{psk5} below),  then  we deduce that the function $t\mapsto \inf_D z_t$ cannot  stay away from zero on $[0,\tau_1]$ when $\tau_1<\infty.$ Consequently, $z_t$ also quenches in finite time on the event $\left\lbrace \tau_1 < + \infty\right\rbrace$  and $\tau_1$ is an upper bound for the quenching time of $z_t.$

 In the sequel we are working towards the estimation of the probability of the event $\left\lbrace \tau_1 = + \infty\right\rbrace,$ so   we have
  \begin{eqnarray}
 \mathbb{P}[\tau_1=+\infty]&=&
 \mathbb{P}\left[\int_0 ^t e^{3\kappa W_s+3(\lambda_1+\frac{\kappa ^2}{2})s} ds < \frac{1}{3\lambda}B_0^3, \quad \mbox{for all} \quad t>0 \right]\nonumber\\
 &=&
\mathbb{P}\left[\int_0 ^{+\infty} e^{3\kappa W_s+3(\lambda_1+\frac{\kappa ^2}{2})s} ds \leq
 \frac{1}{3\lambda}B_0^3 \right].
 \end{eqnarray}
 Then by virtue of the law of the iterated logarithm for $W_t,$  cf.  \cite{A95, DKLM}, that is
\bge&&\lim \inf_{t \to + \infty} \frac{W_t}{t^{1/2} \sqrt{2 \log ( \log t)}}=-1, \quad  \mathbb{P}- a.s.\; ,\label{psk4}\\
&&\mbox{and}\no\\
&&\lim \sup_{t \to +\infty} \frac{W_t}{t^{1/2} \sqrt{2 \log ( \log t)}}=+1, \quad  \mathbb{P}- a.s.\; ,\label{psk5}
\ege
we deduce that for any sequence $t_n\to +\infty$
\bgee
W_{t_n} \sim \alpha_n t_n^{1/2} \sqrt{2 \log (\log t_n)},
\egee
with $\alpha_n\in[-1,1],$ and thus
\bgee
\int_0^{+\infty} e^{3\kappa W_s+3(\lambda_1+\frac{\kappa ^2}{2})s}ds=+\infty.
\egee
The latter implies that
\begin{eqnarray}
\mathbb{P} \left[\tau_1=+ \infty \right]= \mathbb{P} \left[ \int_0 ^{+\infty} e^{3\kappa W_s+3(\lambda_1+\frac{\kappa ^2}{2})s}ds \leq \frac{1}{3\lambda} B_0 ^3 \right]=0,\nonumber
\end{eqnarray}
and hence
\begin{eqnarray}
\mathbb{P}\left[\tau_1 <+\infty \right]= 1-\mathbb{P}[\tau_1=+\infty]=1-0=1.
\end{eqnarray}
 Therefore
 $B(t)$ and consequently $ v_t(\phi)$ quenches a.s.  which in turn implies  that   $z_t(\phi)$ quenches a.s. as well. The latter entails, due also to \eqref{eigenc},  that
 \bgee
z_t(\phi)=\int_D  z_t(x)\phi(x)\,dx\geq  \inf_{x\in D}|z_t(x)|
 \egee
 and thus
 \bgee
 \limsup_{t\rightarrow \tau} \inf_{x\in D}|z_t(x)|= 0,
 \egee
 for some $\tau\leq \tau_1$ and independently of the initial condition $z_0$ and the parameter value $\lambda.$
Thus we have the following result.
 \begin{thm} \label{thm2} 
The weak solution of problem  (\ref{LSP1z}) quenches in finite time with probability one, i.e. almost surely,  regardless the size of its initial condition as well as that of parameter $\lambda$.
\end{thm}
  \begin{rem}\label{lpsk1}
  The result of Theorem \ref{thm2} shows that  the  impact of the noise for the dynamics of  problem \eqref{mmk1}  is vital. In particular, the presence of  the nonlinear term $f(z)={z}^{-2}$ forces the solution towards quenching almost surely. In contrast, for the corresponding deteministic problem, i.e. when $k=0,$ and for homogeneous boundary conditions then  quenching occurs only either for large initial data or for  large values of the parameter $\lambda,$ cf. \cite{EGG10, KMS08,  KS18}.
  \end{rem}

\subsection{Introducing  a  regularizing term into model \eqref{LSP1z}}\label{irt}
 A natural question arises is if can modify model \eqref{LSP1z} appropriately so its destructive quenching behaviour can be only limited in a certain range of parameters and so  of global-in-time solutions occur as well. To this end we  consider a model with a  modifiied nonlinear drift term; indeed the drift term $f(z)={z}^{-2},$  which  is responsible for the almost surely quenching (cf. Remark \eqref{lpsk1}, is now multiplied by  $e^{-3\gamma t}$ for $\gamma$ some positive constant.

 Specifically problem  \eqref{LSP1z} now is modified to
\bse\label{LSPg}
 \be\label{LSPg1}
dz_t=(\Delta z_t- \lambda  e^{-3\gamma t}z_t^{-2})dt -\kappa z_t dW_t,\quad
  x\in D, \quad t>0,
 \ee
 \be \label{LSPg2}
\frac{\partial z_t}{\partial \nu}+ \beta z_t=0,\quad x\in \partial D,  \quad t>0, \quad \beta, \kappa,\gamma >0
  \ee
  \be\label{LSPg3}
 0<z_0(x)=z(x,0)\leq 1.
 \ee
\ese


 In the sequel we proceed similarly as in the proof of Theorem \ref{thm2}, so we first set
$$z_t(\phi):= \int_D z_t \phi\, dx\quad \mbox{and}\quad  z_t^{-2}(\phi):=\int_D z^{-2}_t \phi dx$$
where $\phi$ solves \eqref{eigena}-\eqref{eigenc}  and then by second  Green's identity we obtain
\begin{eqnarray}
\Delta z_t(\phi)=
z_t(\Delta \phi),\nonumber
\end{eqnarray}
recalling that
\bgee
z_t(\Delta \phi):=\int_D z_t \Delta \phi \,dx.
\egee
Then the weak formulation of (\ref{LSPg}) is :
\begin{eqnarray}
z_t(\phi)=z_0(\phi) +\int_0^t z_s(\Delta \phi)ds
- \int_0^t \lambda e^{-3\gamma t} z_s^{-2}(\phi)\, ds -
   \kappa  \int_0^t z_s(\phi) dW_s,\;\; \mathbb{P}-\mbox{a.s.}\label{weakf2}
\end{eqnarray}
 We  again consider the stochastic process  $v_t=e^{\kappa W_t}z_t$, for $ 0\leq t<\tau$  with $\tau$ being the stopping (quenching) time of stochastic process $z_t.$  Next using  integration by parts of, see  also \eqref{mmk5} and \eqref{mmk6}, for the stochastic processes
\bgee
z_t=z_0-\kappa \int_0^t z_s dW_s+\int_0^t \left(\Delta z_s-\frac{\la e^{-3\gamma s}}{z_s^2}\right)\,ds
\egee
and for $e^{\kappa W_t}$ given by \eqref{1}
 we obtain that
 \begin{eqnarray}\label{mmk7}
v_t(\phi)=v_0(\phi)+ \int_0^t e^{\kappa W_s} dz_s(\phi)
+ \int_0^t z_s(\phi) d\left( e^{\kappa W_s} \right)
+ \left[ e^{\kappa W_s},z_s(\phi)\right](t)
\end{eqnarray}
where the quadratic variation into \eqref{mmk7} is given by \eqref{mtk1a}.

Therefore, by virtue of  \eqref{LSPg}, \eqref{weakf2} and It\^o's formula, cf. \eqref{2}, we obtain that
\begin{eqnarray}
\label{weakfv} 
v_t(\phi)=v_0(\phi)&+&\int_0^t v_s(\Delta \phi)  ds -\lambda \int_0^t e^{3\kappa W_s} e^{-3\gamma s}v_s^{-2}(\phi)ds
- \frac{\kappa ^2}{2} \int_0^t v_s(\phi)ds,
\end{eqnarray}
taking also into account that $z_t=e^{-\kappa W_t}v_t.$
 
Notably,  via \eqref{weakfv} we deduce that $v_t(x)=v(x,t)$ is a weak solution of the following random PDE
\bse\label{LSPv} 
 \be\label{LSPv1}
\frac{\partial v}{\partial t} (x,t)= \Delta v(x,t)+ \left( \gamma- \frac{\kappa ^2}{2}\right) v(x,t)
+\lambda e^{3\kappa W_t} v^{-2}(x,t),\quad
  \quad\mbox{in}\quad Q_T,
  \ee
  \be\label{LSPv2}
 \frac{\partial v(x,t)}{\partial \nu}+\beta v(x,t)=0,\quad\mbox{on}\quad \Gamma_T,
   \ee
\be\label{LSPv3}
v(x,0)=z_0(x),\quad  x \in D.
  \ee
  \ese
Problem \eqref{LSPv} should be understood   trajectorwise. and its local
 existence, uniqueness and positivity of solution up to eventual quenching time can be derived by \cite[Theorem 9, Chapter 7]{fri}.

 Recalling that  $\phi$ solves  the eigenvalue problem \eqref{eigena}-\eqref{eigenc} then  equation \eqref{weakfv} is reduced to
  \begin{eqnarray} \label{36}
  v_t(\phi)=v_0(\phi)- (\lambda_1  +\frac{\kappa ^2}{2}) \int_0^t v_s(\phi) ds.
  -\lambda \int_0^t e^{-3(\gamma s-\kappa W_s)}\,v_s^{-2} (\phi)ds,
  \end{eqnarray}
 or (cf. subsection \ref{bp})  in differential form
   \begin{eqnarray}
   \frac{d v_t(\phi)}{dt}= - \left( \lambda_1 +\frac{\kappa ^2}{2} \right) v_t(\phi) -
   \lambda e^{-3(\gamma t-\kappa W_t)}  v_t^{-2}(\phi).
   \nonumber
    \end{eqnarray}
   Next  by  virtue of  Jensen's inequality we deduce
     \begin{eqnarray}
  \frac{d v_t(\phi)}{dt} \leq - \left( \lambda_1 + \frac{\kappa ^2}{2} \right)  v_t(\phi)- \lambda e^{-3(\gamma t-\kappa W_t)}  (v_t(\phi))^{-2},
 \quad v_0(\phi)>0.
 \nonumber
    \end{eqnarray}
     By comparison we get  $v_t(\phi)\leq \Psi(t)$ where $\Psi(t)$ satisfies the following Bernoulli differential equation
   \begin{eqnarray}
  \Psi'(t)=- \left( \lambda_1 + \frac{\kappa ^2}{2} \right)   \Psi(t)- \lambda e^{-3(\gamma t-\kappa W_t)}  \Psi^{-2} (t),
 \quad \Psi_0= \Psi(0)=v_0(\phi)>0,
\nonumber
    \end{eqnarray}
  with solution
   \begin{eqnarray}
    \Psi(t)=e^{ -  \left( \lambda_1+\frac{\kappa ^2}{2} \right)t}
     \left[ \Psi_0^3 -
    3\lambda \int_0^t e^{ 3 \left( \lambda_1- \gamma+\frac{\kappa ^2}{2} \right)s+3 \kappa W_s} ds \right]^{\frac{1}{3}}, \quad 0\leq t < \tau,
    \nonumber
    \end{eqnarray}
      with
     $$\tau_2:= \inf \left\lbrace t\geq 0 : \int_0^t e^{ 3 \left( \lambda_1- \gamma+
     \frac{\kappa ^2}{2} \right)s+3 \kappa W_s} ds\geq \frac{1}{3\lambda}\Psi_0^3 \right\rbrace,$$
   being the stopping time of $\Psi(t).$

 It follows that $\Psi(t)$ extincts to zero in finite time on the event $\left\lbrace \tau_2 < +\infty \right\rbrace$.
  Since $v_t(\phi)\leq \Psi(t)$ then $\tau_2$ is an upper bound for the stopping (extinction) time $\tau$ of $v_t(\phi),$ which is also the stopping (quenching) times of $v_t$ and $z_t.$ 

More specifically we have
 \begin{eqnarray}\label{mtk1}
 \Pbb\left[\tau_2=+\infty \right]&=&\Pbb\left[\int_0^t e^{3\kappa W_s+3 \left(\lambda_1- \gamma + \frac{\kappa ^2}{2}\right)s}ds < \frac{1}{3\lambda}\Psi^3_0 , \quad \mbox{for all} \quad t>0 \right]\nonumber\\
  &=&\Pbb\left[\int_0^{+\infty}  e^{3\kappa W_s+3 \left(\lambda_1- \gamma
  + \frac{\kappa ^2}{2}\right)s}ds \leqq \frac{1}{3\lambda}\Psi^3_0 \right].
     \end{eqnarray}
Then via the change of variables $s_1\mapsto \frac{ 9 \kappa^2 s}{4}$  and making use of the scaling property of  $W_t$ we obtain
\bge\label{mtk69}
 \Pbb\left[\tau_2=+\infty \right]= \Pbb\left[ \frac{4}{9 \kappa ^2 } \int_0^{+\infty}  e^{2 W_{ s_1}+\frac{4}{3 \kappa^2}\left(\lambda_1- \gamma
  + \frac{\kappa ^2}{2}\right)s_1}ds_1 \leqq \frac{1}{3\lambda}\Psi^3_0 \right].
\ege
Setting  $W^{(\mu)}_s:=W_s+\mu s$,  with
 $\mu :=\frac{2}{3 \kappa^2}\left(\lambda_1- \gamma
  + \frac{\kappa ^2}{2}\right)$
 then \eqref{mtk69} reads
 \begin{eqnarray}\label{mtk71}
  \Pbb\left[\tau_2=+\infty \right]= \Pbb\left[ \frac{4}{9 \kappa ^2 }\int_0^ {+\infty} e^ {2 W^{(\mu)}_s} ds \leqq \frac{1}{3\lambda}\Psi^3_0 \right].
\end{eqnarray}
    We now distinguish two cases:
    \begin{enumerate}
    \item
    We first take $\gamma\geq \lambda_1+\frac{\kappa ^2}{2}$
    and  thus we have
  \begin{eqnarray}\int_0^{\infty} e^{2 W_s^{(\hat{\mu})}} ds= \frac{1}{2 Z_{-\hat{\mu}}},
  \nonumber
    \end{eqnarray}
 cf. see \cite[ Chapter 6, Corollary 1.2]{YOR},   where $Z_{-\mu}$ is a random variable with law $\Gamma(-\mu),$ i.e.
  $$\Pbb\left(Z_{-\mu} \in dy\right)=\frac{1}{\Gamma(-\mu)} e^{-y} y^{-\mu-1}\,dy,$$
 where $\Gamma(\cdot)$ is the complete gamma function, cf. \cite{AS}.

Hence \eqref{mtk71} entails (see also  in \cite[formula 1.104(1) page 264]{BS02})
\bge \label{LSPgPropG}
     \Pbb\left[ \tau_2 = +\infty \right]=\int_0^{ \frac{1}{3\lambda}v^{3}_0(\phi) } \frac{\left( \frac{9\kappa ^2 y}{2}\right ) ^{- \frac{\left( \kappa ^2- 2\gamma +2 \lambda_1\right)}{3 \kappa ^2}}}{y \Gamma \left( - \frac{ \left( \kappa ^2- 2\gamma +2 \lambda_1\right)} {3\kappa ^2 } \right)}  \exp \left(- \frac{2}{9 \kappa ^2 y} \right)\, dy,
\ege
hence
\bgee
 \Pbb\left[ \tau_2 <+\infty \right]=1- \Pbb\left[ \tau_2 = +\infty \right]=\int_{ \frac{1}{3\lambda}v^{3}_0(\phi) }^{+\infty} \frac{\left( \frac{9\kappa ^2 y}{2}\right ) ^{- \frac{\left( \kappa ^2- 2\gamma +2 \lambda_1\right)}{3 \kappa ^2}}}{y \Gamma \left( - \frac{ \left( \kappa ^2- 2\gamma +2 \lambda_1\right)} {3\kappa ^2 } \right)}  \exp \left(- \frac{2}{9 \kappa ^2 y} \right)\, dy.
\egee
Now since $\tau<\tau_2$ we have that
\bge\label{LSPgProp}
 \Pbb\left[ \tau <+\infty \right]\geq \int_{ \frac{1}{3\lambda}v^{3}_0(\phi) }^{+\infty} \frac{\left( \frac{9\kappa ^2 y}{2}\right ) ^{- \frac{\left( \kappa ^2- 2\gamma +2 \lambda_1\right)}{3 \kappa ^2}}}{y \Gamma \left( - \frac{ \left( \kappa ^2- 2\gamma +2 \lambda_1\right)} {3\kappa ^2 } \right)}  \exp \left(- \frac{2}{9 \kappa ^2 y} \right)\, dy.
 \ege
\item  Next we assume that $\mu>0, $  i.e.   $\gamma< \lambda_1+\frac{\kappa ^2}{2}$.
 Then using the  law of the iterated logarithm, cf.  \eqref{psk4} and \eqref{psk5}, for $W_t$, we obtain
 \bgee
\int_0^{+\infty}  e^{3\kappa W_s+3 \left(\lambda_1- \gamma
  + \frac{\kappa ^2}{2}\right)s}ds=+\infty,
 \egee
   hence via \eqref{mtk1} we derive
   \begin{eqnarray}
\mathbb{P} \left[\tau_2=+ \infty \right]= \mathbb{P} \left[ \int_0 ^{+\infty} e^{3\kappa W_s+3(\lambda_1+\frac{\kappa ^2}{2})s}ds \leq \frac{1}{3\lambda} \Psi_0 ^3 \right]=0
  \nonumber
    \end{eqnarray}
and thus
 \begin{eqnarray}
 \Pbb \left[ \tau_2<+ \infty \right]=1-\mathbb{P} \left[\tau_2=+ \infty \right]=1.
  \nonumber
    \end{eqnarray}
\end{enumerate}
Summarizing  the above we have the following result

\begin{thm} \label{thm3a} 
\begin{enumerate}
\item  If   $\gamma\geq \lambda_1+ \frac{\kappa ^2}{2}$  then the weak  solution of problem  (\ref{LSPg})  quenches in finite time with
probability bounded below as shown in \eqref{LSPgProp}.
\item In the complementary case when
$\gamma<\lambda_1+ \frac{\kappa ^2}{2}$  then the weak  solution of problem  (\ref{LSPg})  quenches in finite time almost surely.
\end{enumerate}
\end{thm}
\begin{rem}\label{aek1}
Let us fix $\gamma$ and $\kappa$ so that $\gamma- \frac{\kappa ^2}{2}>0.$ Then Theorem \ref{thm3a}(ii) entails that  quenching behaviour dominates when $\la_1$  is  big which only occurs when the domain $D$ is rather small.
\end{rem}

In Figure \ref{FigP_2} an upper bound of the probability of  global existence, 
 provided by \eqref{LSPgPropG},  is displayed  with respect to the parameter $\lambda$ in Figure  \ref{FigP_2}(a) and with respect to the parameter $a$ in Figure  \ref{FigP_2}(b). In that case an initial condition of  the form  $z_0(x)=1-ax(1-x)$ is considered. Specifically, in Figure  \ref{FigP_2}(a) we observe a decrease of the probability of global existence, as $\lambda$ increases. Similarly in Figure \ref{FigP_2}(b) again reducing the minimum of the initial condition  results
 in decreasing the probability of global existence and this becomes more intence as $\lambda$ increases.
\begin{figure}[!htb]\vspace{-5cm}\hspace{-2.5cm}
   \begin{minipage}{0.6\textwidth}
     \centering
     \includegraphics[width=1.2\linewidth]{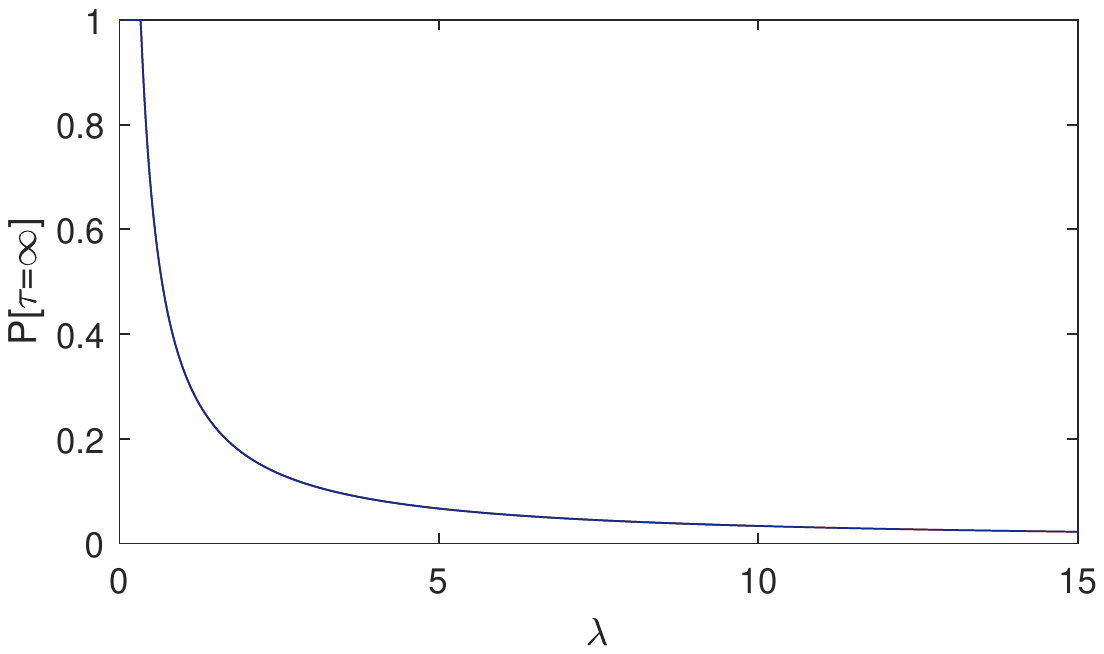}
       \put(-170,300){(a)}
   \end{minipage}\hspace{-2cm}
   \begin{minipage}{0.6\textwidth}
     \centering
     \includegraphics[width=1.2\linewidth]{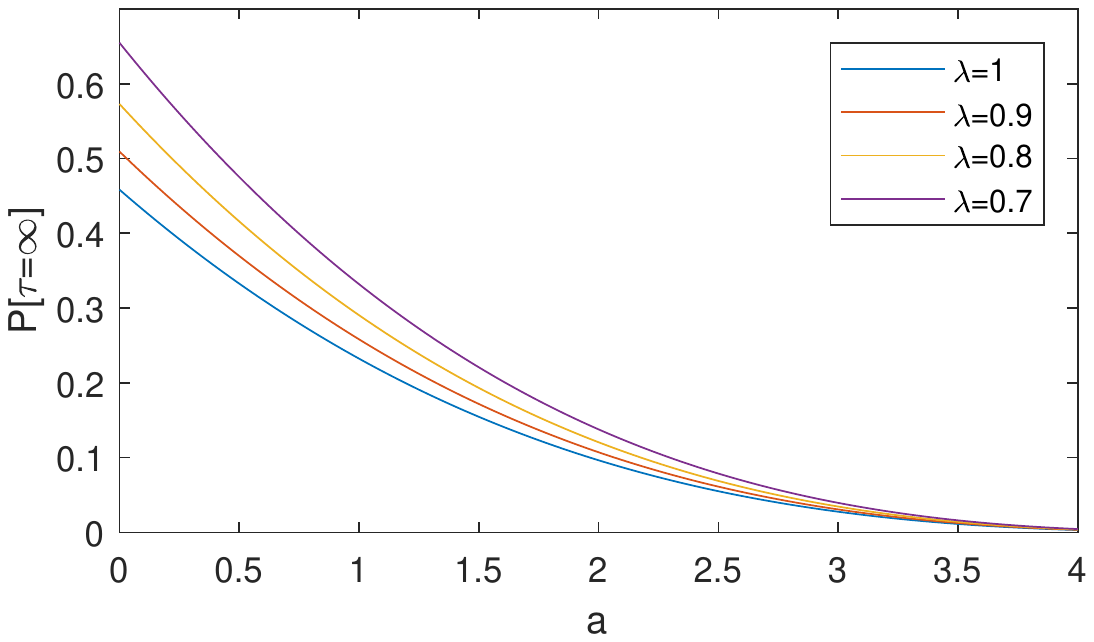}
       \put(-170,300){(b)}
   \end{minipage} \vspace{-6cm}
   \caption{(a) Diagram of the probability $\mathbb{P}\left[\tau =+\infty \right]$ with respect to the parameter $\lambda $,
   (b) with respect to the parameter $a$ in the initial condition for various values of the parameter $\lambda$.}\label{FigP_2}
\end{figure}
Besides, in Figure \ref{FigP_2a}  the behaviour of the probability of 
the global existence, bounded above by the quantity defined in \eqref{LSPgProp},  is examined with respect to the parameter $\gamma$ and the noise amplitude $\kappa.$ 
In particular, the impact of  parameter $\gamma,$ i.e. the coefficient of the regularizing term,   is displayed  in  Figure \ref{FigP_2a}(a). Note that the condition $\gamma>\lambda_1+\kappa^2$ should be satisfied (here $\lambda_1=\pi^2$  and 
$\kappa=1$);  then  we observe a peak of the probability at the value $\gamma = 13.77.$
  Moreover in Figure \ref{FigP_2a}(b)  the variation of that  probability with respect  to the parameter $\kappa$ for various values of the parameter $\lambda$ is shown. In that case a similar peak is attained at the value 
  $\kappa=1.084.$
  
\begin{figure}[!htb]\vspace{-5cm}\hspace{-2.5cm}
   \begin{minipage}{0.6\textwidth}
     \centering
     \includegraphics[width=1.2\linewidth]{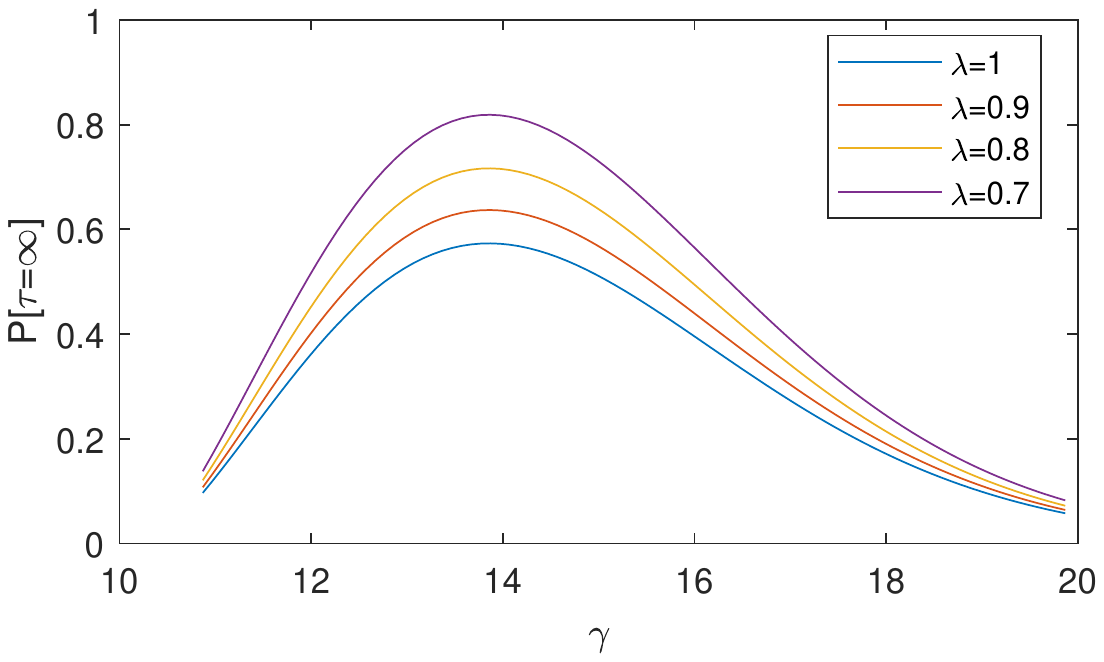}
       \put(-170,300){(a)}
   \end{minipage}\hspace{-2cm}
   \begin{minipage}{0.6\textwidth}
     \centering
     \includegraphics[width=1.2\linewidth]{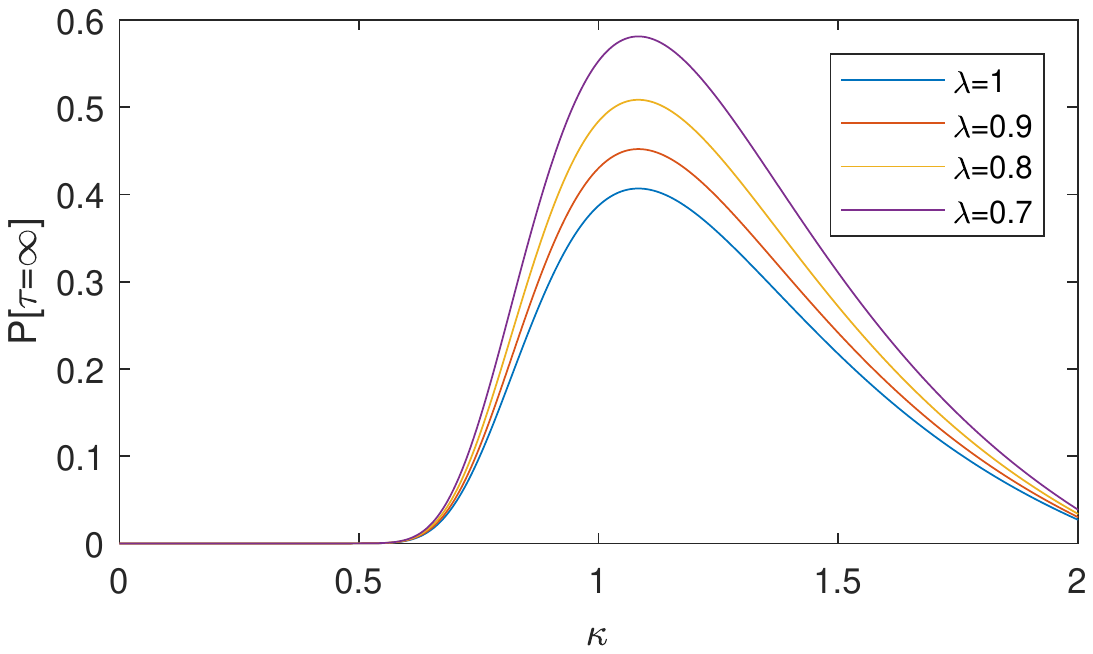}
       \put(-170,300){(b)}
   \end{minipage} \vspace{-6cm}
   \caption{ Diagram of the probability  $\mathbb{P}\left[\tau =+\infty \right]$ for various values of the parameter $\lambda$, (a) with respect to the parameter $\gamma $,
  (b) with respect to the  noise amplitude $\kappa$  }
    \label{FigP_2a}
\end{figure}
\subsection{ {Model \eqref{GSPzu}}}
 In the current subsection we investigate the probability of quenching for the solution of problem \eqref{GSPzu}
    where  $g, \kappa_1: \mathbb{R}_+\rightarrow \mathbb{R}_+$  and
   $h: D\times \mathbb{R}_+ \rightarrow \mathbb{R}_+$ are continuous funtions. Note  that  from mathematical modelling perspective the function $g(t)$ represents the dispersion coefficient  whilst  $h(x, t)$ describes the varying dielectric properties of the elastic membrane (\cite{ FMPS07}), cf. section \ref{cmm}.

Next we define the random  process
\[
M_t:=\int_0^t \kappa_1(s)dW_s,
\]
and we set
\begin{equation}\label{57}
v_t:= e^{M_t} z_t,\quad 0\leq t<\tau,
\end{equation}
where again $\tau$ is the stopping (quenching) time of stochastic process $z_t$ determined by \eqref{GSPzu}.

In the sequel we  proceed as in \cite{AJN}. It\^o's formula implies the semimartingale expansion
\bge\label{mmk71}
e^{M_t}=1+\int_0^t \kappa_1(s) e^{M_s} dW_s + \frac{1}{2} \int_0^t \kappa_1^2(s) e^{M_s} ds .
\ege
Next by letting $z_t(\phi):= \int_D z_t \phi dx$ and $z_t^{-2}(\phi):=\int_D z^{-2}_t \phi dx,$ where again $\phi \in C^2(D)$ solves the eigenvalue problem \eqref{eigena}-\eqref{eigenc},  we have
\bge
z_t(\phi)= z_0(\phi)+  \int_0^t g(s) \Delta z_s(\phi) ds 
- \lambda\int_0^t h(x,s) z_s^{-2}(\phi) ds -\int_0^t \kappa_1(s) z_s(\phi)dW_s,\label{58}
\ege
$\mathbb{P} - a.s.$ for all $ t \in [0, \tau).$ 

Note also that for  any fixed $\phi$, the process $(z_t(\phi)1_{[0,\tau)}(t))_{t\in \mathbb{R}_+}$ is also a semimartingale. Moreover  using  integration by parts formula,  cf.  \eqref{mmk5} and \eqref{mmk6},  we get the weak formulation
\begin{eqnarray}
v_t(\phi) &=& e^{M_t} z_t(\phi) \nonumber \\
&=&e^{M_0}z_0(\phi)+ \int_0^t e^{M_s} d(z_s(\phi))+\int_0^t z_s(\phi) d(e^{M_s})
+ \left[e^{M_t},z_t(\phi)\right]\,\label{59}
\end{eqnarray}
where the quadratic variation (see  \cite[ section 7.6, pg. 113]{1}) is given by
$$
\left[e^{M_t}, z_t(\phi)\right](t):=- \int_0^t \kappa^2 _1 (s) e^{M_s}z_s(\phi) ds.
$$
 Next  \eqref{59} in conjunction with \eqref{57}, \eqref{mmk71} and  \eqref{58} yields
\begin{eqnarray}
v_t(\phi) &=&  z_0(\phi) +
 \int_0^t e^{M_s} \left( g(s) \Delta z_s(\phi)  -\lambda z_s^{-2}(h\phi) \right) ds  \nonumber \\
&&\quad\quad \quad + \int_0^t e^{M_s} \kappa_1 (s) z_s(\phi) dW_s
- \int_0^t e^{M_s} \kappa_1 (s) z_s(\phi) dW_s  \nonumber \\
&& \quad\quad \quad + \frac{1}{2} \int_0^t \kappa_1 ^2 (s) e^{M_s} z_s(\phi) ds  - \int_0^t  \kappa^ 2_1(s) e^{M_s} z_s(\phi) ds \nonumber\\
&=& v_0(\phi) +  \int_0^t g(s) v_s(\Delta \phi) ds
 - \lambda\int_0^t  e^{3M_s} v_s^{-2}(h\phi) ds-\frac{1}{2} \int_0^t \kappa_1 ^2(s) v_s(\phi) ds\nonumber\\
&=& v_0(\phi) -\lambda_1\int_0^t g(s) v_s(\phi) ds-\lambda\int_0^t  e^{3M_s} v_s^{-2}(h\phi) ds
-\frac{1}{2} \int_0^t \kappa_1 ^2(s) v_s(\phi) ds ,\label{60}
\end{eqnarray}

where
\bgee
v_s^{-2}(h\phi) :=\int_D v_s^{-2}(x) h(x,t) \phi(x)\,dx,
\egee
taking also into account that  $z_0(\phi)= v_0(\phi)$  due to  (\ref{57})  as well as that  $\Delta v_s (\phi)=v_s(\Delta \phi)=-\lambda_1 v_s(\phi)$ via Green's second identity.

 Equation \eqref{60} can then be written in differential form as 
\bgee
\frac{d v_t(\phi)}{dt}=-\left(\lambda_1 g(t)+\frac{1}{2}\kappa_1^2(t)\right) v_t(\phi)-\lambda e^{3M_t} v_t^{-2}(h\phi),
\egee
 which by virtue of  Jensen's inequality  infers
\bgee
\frac{d v_t(\phi)}{dt}\leq-\left(\lambda_1 g(t)+\frac{1}{2}\kappa_1^2(t)\right)v_t(\phi)-\lambda \omega e^{3M_t}(v_t(\phi))^{-2},
\egee
where  $\omega:=\max_{(x,s)\in \bar{D}\times [0,\tau]} h(x,s)>0$
 and by means of a comparison argument
 we get
\begin{equation}
v_t(\phi) \leq A(t), \quad 0\leq t<\tau , \label{67}
\end{equation}
where now $A(t)$ denotes the solution of the initial value problem
\begin{eqnarray}
\label{68}
A'(t)=-\left(\lambda_1 g(t)+\frac{1}{2}\kappa_1^2(t)\right) A(t)
-\lambda \omega e^{3M_t} A^{-2}(t) , \; 0<t<\tau,
\quad
A_0=A(0)= v_0(\phi)>0,
 \nonumber
\end{eqnarray}
 with solution
   \begin{equation}
  \label{70}
  A(t) =e^{-\left( \lambda_1 K(t) +\frac{1}{2} J(t)\right)}
  \left[ v_0(\phi) ^3- 3 \lambda\omega \int_0^t e^{3M_s +3( \lambda_1 K(s) + \frac{1}{2}J(s))} ds \right]^{\frac{1}{3}},
  \end{equation}
  where  $K(t):= \int_0^t g(s) ds$ and $J(t) : = \int_0^ t \kappa_1^2 (s) ds.$

  The maximum existence (stopping) time $\tau_3$ of $A(t)$ is then given by
  \[ \tau_3:=\left\lbrace t\geq 0  : \int_0^t e^{3M_s + 3(\lambda_1 K(s) + \frac{1}{2}J(s))} ds \geq  \frac{1}{3\lambda\omega} v^3_0(\phi) \right \rbrace,
  \]
and actually  $A(t)$ quenches  in finite time on the event $\left\lbrace \tau_3 < + \infty\right\rbrace.$
The fact  that $0\leq v_t(\phi)\leq A(t)$ reveals that $\tau_3$ is an upper bound of the stopping (extinction) time $\tau$ for $v_t(\phi),$ hence the function
\[
t\mapsto\int_D e^{M_t} z_t(x)\phi(x)\,dx
\]
quenches in finite time under the event  $\left\lbrace \tau_3< + \infty\right\rbrace.$  Using now \eqref{eigenc} as well as the fact that $t\mapsto e^{M_t}$  is bounded below  away from zero (cf. \eqref{psk4}, \eqref{psk5} and the fact that $\kappa_1(t)$ is bounded ) on $[0,\tau_3],$ once $\tau_3<\infty,$  then we deduce that the function 
$t\mapsto \inf_D z_t $ cannot stay away from zero on  $[0,\tau_3]$ for $\tau_3<\infty.$ Therefore, $z_t$ quenches in finite time on the event $\left\lbrace \tau_3 < + \infty\right\rbrace$ and so $\tau_3$  is an upper bound for the quenching time of $z_t$.

Observe that $M_t= \int_0^t \kappa_1 (s) dW_s$ is a continuous martingale and so it can be written as a time - changed Brownian motion $M_t= W_{J(t)},$ where $J(t)=[M](t)= \int_0^t \kappa^2 _1 (s) ds $
 is the quadratic variation of $M,$ cf. \cite[Theorem 4.6 page 174]{KS91} and \cite{AJN}.

 Set  $\rho:=\frac{1}{3 \lambda\omega} v^3_0(\phi)$ then
 \begin{eqnarray}
   \mathbb{P}(\tau_3=+ \infty) &=&
   \mathbb{P} \left(  \int_0^t e^{3M_s + 3(\lambda_1 K(s)+ \frac{1}{2}J(s))} ds
   < \frac{1}{3 \lambda\omega} v^3_0(\phi),
    \quad \mbox{for all} \quad t>0 \right)  
     \nonumber\\
  & =& \mathbb{P}\left( \int_0^{+\infty} e^ {3 W_{J(s)} + 3(\lambda_1 K(s)+ \frac{1}{2}J(s)) } ds
  \leq \rho\right)
   \nonumber \\
   &=&  \mathbb{P} \left( \int_0^{+\infty} \frac{1}{\kappa_1 ^2 (J^{-1} (s_1))} e^{3 W_{s_1}+3 \left( \lambda_1 K(J^{-1}(s_1))+\frac{1}{2} s_1 \right) } ds_1
  \leq \rho \right) \label{Probgt}
  \end{eqnarray}
  where $s_1:=J(s).$

  At that point we introduce the assumption that coefficients $g(t)$ and $\kappa_1(t)$ satisfy:  there exists some positive constant $C$  such that
  \begin{equation}
  \label{assum71}
  \frac{1}{\kappa_1 ^2 (t)} e^{ 3\lambda_1\left( K(t) +\frac{1}{2}J(t) \right) } \geq C\quad\mbox{for any}\quad t\geq 0.
  \end{equation}

 Then \eqref{Probgt} via  \eqref{assum71} reads
 \bge \label{assum71a}
 \mathbb{P}(\tau_3=+ \infty)   &\leq&  \mathbb{P} \left( \int_0^{\infty}  e^{3 W_{s_1}+\left(-\frac{3}{2} \lambda_1  J(J^{-1}(s_1) + \frac{3}{2}  s_1\right) } ds_1 \leq  \frac{\rho}{ C} \right)\nonumber\\
 &=& \mathbb{P} \left( \int_0^{\infty}  e^{3 W_{s_1}+\frac{3}{2}\left( 1 -\lambda_1\right) s_1 } ds_1 \leq  \frac{\rho}{ C} \right).
 \ege
  Next  we introduce the change of variables  $s_2\mapsto \left(\frac{3}{2} \right) ^2 s_1 ,$  and thus  again via
  the scaling property of  $W_t$ then  \eqref{assum71a} entails
\begin{eqnarray}
  \mathbb{P}\left( \tau_3= + \infty \right)
  &\leq & \mathbb{P}\left( \frac{4}{9} \int_0^{\infty}
   e^{3 W_{\frac{4}{9} s_2}+\frac{3}{2}\left(1-\lambda_1\right) \frac{4}{9}s_2} ds_2 \leq \frac{\rho}{ C} \right)
  \nonumber \\ 
&=&  \mathbb{P} \left( \int_0^{\infty}  e^{2\left(\frac{1-\lambda_1}{3}\right)s_2 + 2 W_{s_2} } ds_2
  \leq \frac{9 \rho}{4  C} \right)\nonumber\\
  &=&\mathbb{P}\left(\int_0^{+\infty} e^{2W_s^{(\mu)}} ds \leq \frac{9 \rho}{4 C} \right),
   \label{Probgts}
  \end{eqnarray}
 where $\mu:= \frac{1-\lambda_1}{3}$ and $W_s^{(\mu)}:=W_s+\mu s.$

 Next we distinguish the following  cases: 
 \begin{enumerate}
 \item Initially we assume  that $\mu<0,$  i.e. $\lambda_1>1.$ Then by virtue of \eqref{Probgts} and following the same reasoning as in subsection \ref{irt} we obtain
 \begin{equation} \label{LSPgtProp} 
  \mathbb{P}( \tau_3= +\infty) \leq \mathbb{P} \left( \frac{1}{2 Z_{-\mu}}\leq \frac{9 \rho}{4C}\right)=\frac{1}{\Gamma(-\mu)} \int_0 ^{\frac{2C}{9 \rho}} y^{-\mu - 1} e^{-y} dy,
  \end{equation}
  cf. \cite[Corollary 1.2 page 95]{YOR}.

 Hence, from (\ref{LSPgtProp}) we derive
   \begin{equation} \label{80}
   \mathbb{P} \left( \tau_3< + \infty \right) =1 - \mathbb{P}( \tau_1= + \infty ) \geq 1-\frac{1}{\Gamma(-\mu)} \int_0 ^{\frac{2C}{9 \rho}} y^{-\mu - 1} e^{-y} dy=\frac{1}{\Gamma(-\mu)} \int_{\frac{2C}{9 \rho}}^{\infty} y^{-\mu - 1} e^{-y} dy.
   \end{equation}
\item In the complimentary case $\mu\geq 0,$  i.e.  when $\lambda_1\leq 1,$ then via the  iterated logarithm law  for $W_s,$ cf. \eqref{psk4} and \eqref{psk5}, we obtain
\[
\int_0^{+\infty} e^{2W_s^{(\mu)}} ds=+\infty
\]
and thus
  $$
   \mathbb{P}[\tau= +\infty]=\mathbb{P}\left(\int_0^{+\infty} e^{2W_s^{(\mu)}} ds \leq \frac{9 \rho}{4 C} \right)=0.
   $$
  The latter implies that
  $$
   \mathbb{P}[ \tau < + \infty] = 1 - \mathbb{P}[\tau= +\infty]=1
   $$
and so in that case  $A(t)$ quenches a.s. independently of the initial condition $v_0$ and the parameter 
$\lambda$, which also entails  that  $v_t$ and  $z_t$ quench as well.
\end{enumerate}
   \begin{thm} \label{thm3}
Assume that condition  \eqref{assum71} holds true for the continuous positive functions  $g(t), \kappa_1(t) >0.$ Then:
\begin{enumerate}
   \item if $\lambda_1>1$  the probability of  quenching of the weak  solution of problem \eqref{GSPzu}  is lower bounded as shown in  \eqref{LSPgtProp},
\item whilst  for  $\lambda_1 \leq 1$  then the weak  solution of problem \eqref{GSPzu} quenches in finite time $\tau<\infty$ almost surely.
\end{enumerate}
\end{thm}
\begin{rem}
Note that in the special case $g(t)=1,\kappa_1(t)=\kappa=$constant and $h(x,t)=1$ then via relation \eqref{Probgt} we recover the result of Theorem \ref{thm2}.
\end{rem}

\begin{rem}
Remarkably Theorem \ref{thm3} (ii) implies that when the diffusion coefficient $g(t)$  is large, enough ensured by condition  \eqref{assum71},  then quenching behaviour dominates for the case of a  big domain $D.$ This looks in the counterintutive to what has been pointed out in Remark \ref{aek1} in the first place, however it  is in full agreement with the phenomenon  observed in \cite{KBM20} where a strong reaction coefficient,  enhanced there by the evolution of  underlying  domain, fights against the development of a singularity. 
\end{rem}

\begin{rem}
Note that  since $K(t)$ and $J(t)$ are increasing functions we have
\bgee
e^{3\lambda_1\left( K(t)+\frac{1 }{2}J(t)\right)}\geq e^{3\lambda_1\left( K(0)+\frac{1 }{2}J(0)\right)}=1,
\egee
and thus condition  \eqref{assum71} holds true provided that $\kappa_1(t)$ is bounded above, i.e. $\sup_{(0,\infty)}\kappa_1(t)=L<\infty.$ In that case we have that  $C=\frac{1}{L^2}.$

Alternatively,  if $\kappa_1(t)$  gets  unbounded as $t\to \infty$ but satisfies the growth condition
\bgee
\frac{d\kappa^2_1(t)}{dt}\leq \beta \kappa^2_1(t), \quad t>0,\quad\mbox{for some}\quad \beta>0,
\egee
then by virtue of L' H\^opital's rule we can show that
\bgee
\lim_{t\to\infty}\frac{e^{J(t)}}{\kappa_1^2(t)}=\infty
\egee
and then using again the monotonicity of $K(t)$ we derive  \eqref{assum71} with $C=1.$
    \end{rem}

In relation to  applications it is of particular interest to simulate the stochastic process describing the operaton of MEMS device and so to investigate  under which circumstances it quenches. For that purpose in the following section we present such a numerical algorithm  together with various related  simulations for problem \eqref{LSP}.

\section{Numerical Solution} \label{na}


\subsection{Finite Elements approximation}
In the current section we present a numerical study of problem (\ref{LSP})  in the one-dimensional   case. 
For that purpose we apply a finite element semi - implicit Euler in time  scheme, cf. \cite{Lord}. 
The  considered noise term is a multiplicative one and of the form $\sigma(u)\,dW_t$ for
 $\sigma(u)=\kappa (1-u)$   with $\kappa>0$. 
We also assume homogeneous Dirichlet boundary conditions at the points $x=0,1,$   although some of the presented numerical experiments also concern homogeneous and nonhomogeneous Robin boundary conditions.
 A homogeneous Dirichlet boundary condition $u(0,t)=u(1,t)=0$  corresponds in having $z=1$ at those points. Remarkably, this is a case is not actually  covered by the  analysis in  section \ref{eqp}.

We apply
a discretization in $[0,T]\times [0,1]$, $0\leq t\leq T$, $0\leq x\leq 1$ with
$t_n=n\delta t$, $\delta t=\left[{T}/{N}\right]$  for $N$ the number of time steps and we also
 introduce the grid points in $[0,1]$, $x_j = j\delta x$, for $\delta x = 1/M$ and $j = 0,1,\ldots, M$.

 Then we proceed with a finite element approximation for problem (\ref{LSP}).
Let $\Phi_j$, $j=1,\ldots, M-1,$ denote the standard linear $B-$ splines on the interval $[0,\,1]$
\begin{eqnarray}
\Phi_j=\left\{\begin{array}{ccc}
 \frac{y-y_{j-1}}{\delta y},\quad y_{j-1}\leq y \leq y_j, \\
 \frac{y_{j+1}\,-y}{\delta y},\quad y_{j}\leq y \leq y_{j+1}, \\
 0,\quad \mbox{elsewhere in}\quad [0,\,1],
\end{array} \right.
\end{eqnarray}
for $j=1,2,\ldots,M-1$. We then set 
$u(x,t)=\sum_{j=1}^{M-1} {a}_{u_j}(t) \Phi_j(x)$, 
$t\geq 0$, $0\leq x\leq 1$.

Substituting the later expression for $u$  into  equation (\ref{LSP1})
 and applying the standard
Galerkin method, i.e. multiplying with $\Phi_i$, for $i=1,2,\ldots, M-1$ and integrating over
$[0,\,1]$, we obtain a system of equations for the ${a_{u_j} }$'s as follows
\begin{eqnarray}\hspace{-.1cm}\label{eqfe}
   \sum_{j=1}^{M-1}  {\dot{a}}_{u_j}  (t)<\Phi_j(x),\,\Phi_i(x)> &= &
 -\sum_{j=1}^{M-1} {a}_{x_j}(t)\left<\Phi'_j(x),\,\Phi'_i(y)\right> \nonumber\\
  &&+  \left<F\left(\sum_{j=1}^{M-1} {a}_{u_j}(t) \Phi_j(x)\right),\,\Phi_i(x)\right>,\quad\quad
 \nonumber\\
 &&+  \left<\sigma \left(\sum_{j=1}^{M-1} {a}_{u_j}(t) \Phi_j(x)\right) dW(x,t),\,\Phi_i(x)\right>,\quad\quad
\end{eqnarray}
where $<f,g>:=\int_0^1 f(x)g(x)dx$ and $i=1,2,\ldots,M-1$, and in our case 
$F(s)=\frac{\lambda}{(1-s)^2}$, $\sigma(s)=\kappa\,(1-s)$.

 Setting $a_u=[a_{u_1},\,a_{u_2},\ldots,a_{u_{M-1}}]^T$
    the system of equations for the ${a_u}$'s take the form
    \begin{eqnarray}
A {\dot{a}_u}(t)&=&-B a_u(t) +b(t) + b_s(t),\nonumber
\end{eqnarray}
 for 
 \begin{eqnarray*}
 b(u)=\left\{\left<F \left(\sum_{j=1}^{M-1} {a}_{u_j}(t) \Phi_j(x)\right),\,\Phi_i(x)\right>\right\}_i,\\ 
 b_s(u,\,\Delta W_t)=
 \left\{\left<\sigma \left(\sum_{j=1}^{M-1} {a}_{u_j}(t) \Phi_j(x)\right) \Delta W(x,t),\,\Phi_i(x)\right>\right\}_i,
 \end{eqnarray*} 
  the latter  coming from the corresponding  It\^o integral, and 
  $d W_t\simeq \Delta W_h(x,t)=W_h(t+\delta t,x)-W_h(t,x)$ for $W_h(t)$ the finite sum giving  the discrete approximation of $W(t)$.
  
   More specifically the approximation $W_h$ should have the form  $W_h(t):=\sum_{j=1}^{M-1} \sqrt{q_j}\chi_j \beta_j(t)$.
Additionally in order to obtain the same sample path $W(t)$ with different time steps we use the reference time step $\delta t_{r}=T/(m N)$, $m\in \mathbb{N}^+$. 
The increments over intervals of size $\delta t=m \delta t_{r}$   are given by
\[W_h(t+\delta t)-W_h(t)=\sum_{n=0}^{m-1}  W_h(t+t_{n+1})-W_h(t+t_n).  \]

Moreover we approximate the space-time white noise by taking
   \[W_h(t^{n+1})-W_h(t^{n}) =\sqrt{\delta t_r}\sum_{j=1}^{ M-1}\sqrt{q_j}\chi_j\xi_j^n,\]
  where $\xi_j^n:=( \beta_j(t_{n+1})-\beta_j(t_{n}))/\sqrt{\delta t_r}$ and  $\xi_j^n\sim N(0,1)$ are i.i.d. random variables
   for i.i.d.  Brownian motions $\beta_j(t)$.
   Also the eigenfunctions $\chi_j=\chi_j(x)=\sqrt{2}\sin\left( j\pi x\right)$,  $j\in\mathbb{N}^{+}$  are taken as a basis of $L^2(0,1)$ and
   $q_j$ are chosen to be
   \begin{eqnarray}
   q_j=\left\{
   \begin{array}{cc}
   l^{-(2r+1+\epsilon)} & j=2l+1, \, j=2l,\\
   0         \quad \quad     \quad \quad               & j=1,
   \end{array}\right.
   \end{eqnarray}
   for $l\in \mathbb{N}$, $r$ being the regularity parameter, $0\ll \epsilon <1$ to obtain an 
   $H_0^r(0,1)$-valued process.


    We then apply a semi-implicit Euler method in time by taking
\[ A{\dot{a}_u(t_{n})}\simeq A\left({a_u^{n+1}-a_u^{n}}\right)/({\delta t})=
 -B a_u^{n+1} + b(u^{n})+b_s(u^n)\]    
   or 
   \[\left( A+\delta t B\right)a_u^{n+1} =a_u^{n}+\delta t\, b(u^{n})+\delta t \,b_s(u^n,\,\Delta W_h^n)\]
   with  the${(M-1)\times(M-1)}$ matrices $A,B$  having the form
 {\small
\begin{eqnarray}\hspace{-1.4cm}
&& A=\delta x \left[\begin{array}{ccccc}
 \frac23 & \frac16 & 0 & \ldots & 0 \\
  \frac16 & \frac23 & \frac16 & \ldots & 0\\
  0 & 0 & \ddots & \ddots & 0\\
  0 & 0 & \ldots &\frac16 & \frac13
\end{array} \right], \,
B=\frac{1}{\delta y} \left[\begin{array}{ccccc}
 2 & -1 & 0 & \ldots & 0 \\
  -1 & 2 & -1 & \ldots & 0\\
  0 & 0 & \ddots & \ddots & 0\\
  0 & 0 & \ldots &-1 &  2
\end{array} \right], \,         \nonumber      \\
&& b^n=b(u^n)=\left\{\left<F \left(\sum_{j=0}^{M-1} {a}_{u_j}^n \Phi_j(x)\right),\,\Phi_i(x)\right>\right\}_i,
 \nonumber      \\
&& b_s^n= b_s(u^n,\,\Delta W_h^n)=
 \left\{\left<\sigma \left(\sum_{j=0}^M {a}_{u_j}^n\Phi_j(x)\right) \Delta W_h^n,\,\Phi_i(x)\right>\right\}_i \nonumber,
\end{eqnarray}
} 
for ${a}_{u_j}^n={a}_{u_j}(t_n)$, $i=1\ldots, M-1$.
 
 Finally the corresponding algebraic system for the $a_u^n$'s after some manipulation becomes
 \begin{eqnarray}\label{numscha1}
 a_u^{n+1}=\left(A+\delta t B \right)^{-1}\left[a_u^{n} + \delta t\, b^n+ \delta t\, b_s^n\right],
\end{eqnarray}
 for $a_u^1$  being determined by the initial condition.  
    

\subsection{Simulations} 
 
 Initially we present a realization of the numerical solution of problem \eqref{LSP} in Figure \ref{Fig_sim1}(a) for 
 $\lambda=1$, $\kappa=1$,    initial condition $u(x,0)=c\,x(1-x)$ for $c=0.1$ and homogeneous Dirichlet boundary conditions ($\beta\to\infty$, $\beta_c=0$). By this performed realization the occurrence of  quenching is evident.
 For a different realization but for the same parameters in Figure \ref{Fig_sim1}(b) the maximum of the solution at each time step is plotted and again a similar quenching behaviour is observed.
\begin{figure}[!htb]\vspace{-4cm}\hspace{-1cm}
   \begin{minipage}{0.48\textwidth}
     \centering
     \includegraphics[width=1.2\linewidth]{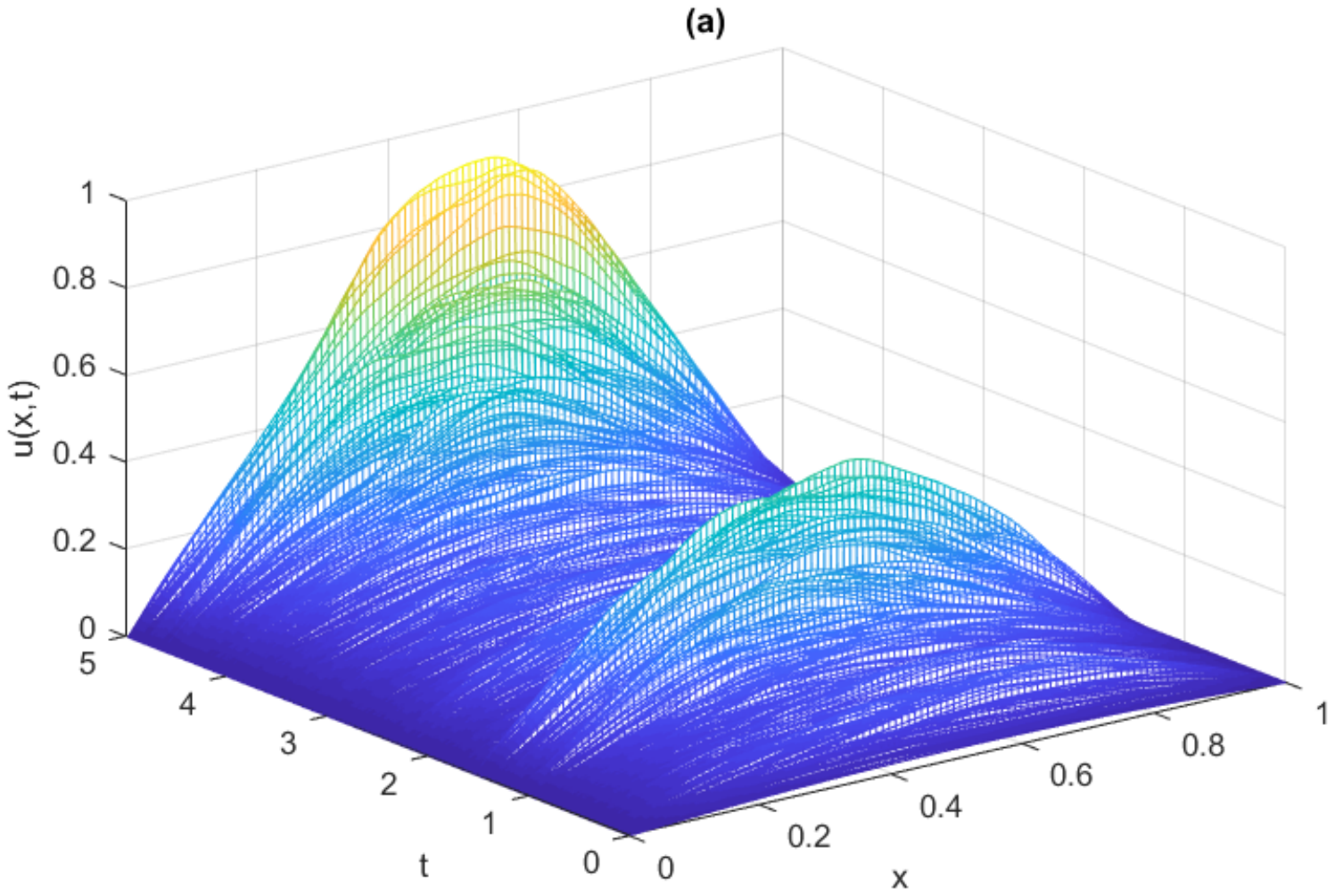}
   \end{minipage}\hfill
   \begin{minipage}{0.48\textwidth}\hspace{-2cm}
     \centering
     \includegraphics[width=1.2\linewidth]{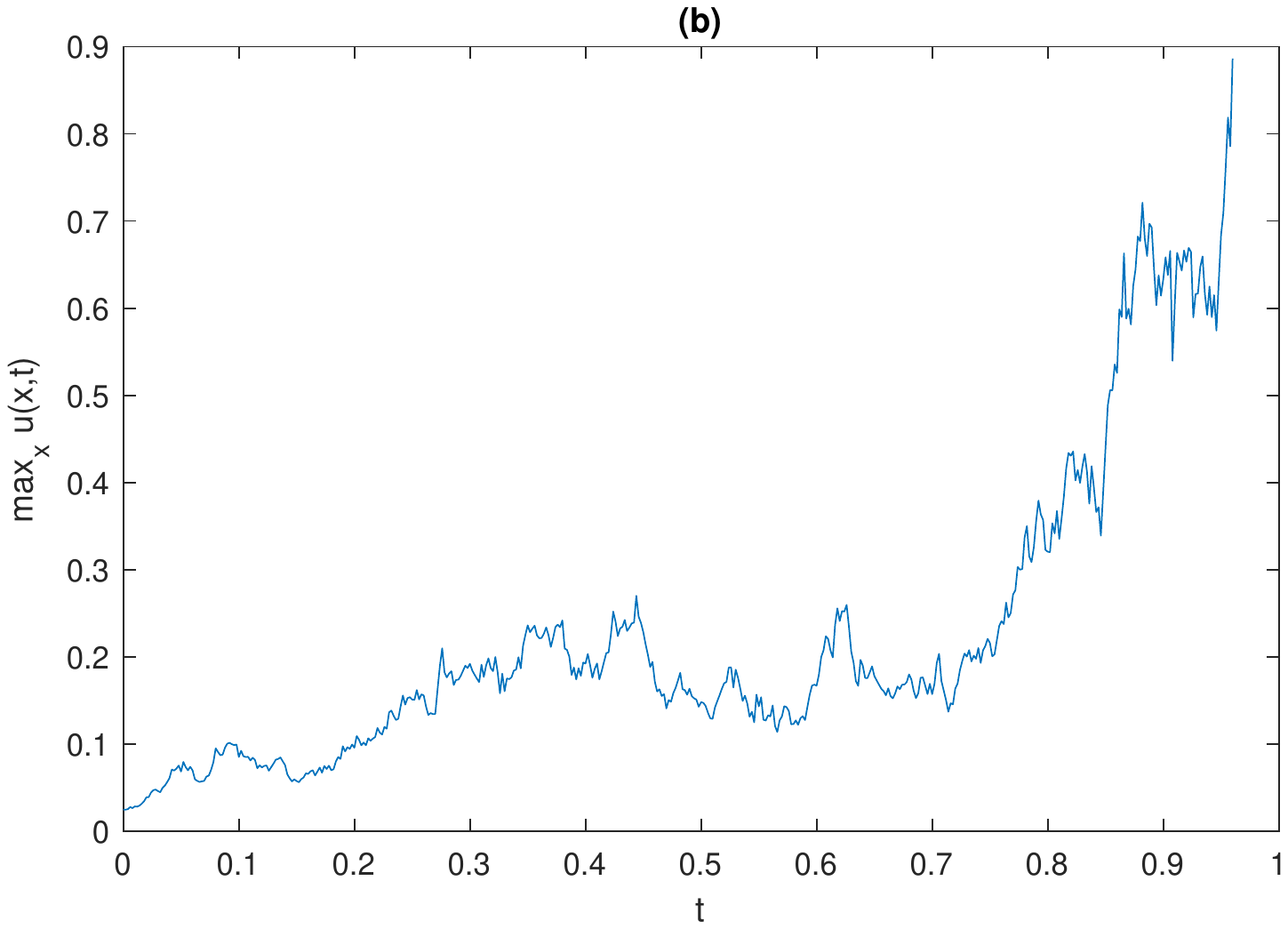}
   \end{minipage}\vspace{-4cm}
   \caption{(a) Realisation of the numerical solution of problem \eqref{LSP} for $\lambda=1$, $k=1$, $M=102$, $N=10e4$,  $r=0.1$  and initial condition $u(x,0)=c\,x(1-x)$ for $c=0.1$.
   (b) Plot of $\|u(\cdot,t) \|_\infty$ from a different realization but with the same   parameters values}.\label{Fig_sim1}
\end{figure}
Next  in Figure \ref{Fig_sim2}(a) we observe the quenching behaviour of five realisations of the numerical solution of problem \eqref{LSP} for $\lambda=2$. In an extra realization depicted in Figure \ref{Fig_sim2}(b)  the spatial distribution of the numerical solution at different time instants can be seen.
\begin{figure}[!htb]\vspace{-4cm}\hspace{-.5cm}
   \begin{minipage}{0.48\textwidth}
     \centering
     \includegraphics[width=1.2\linewidth]{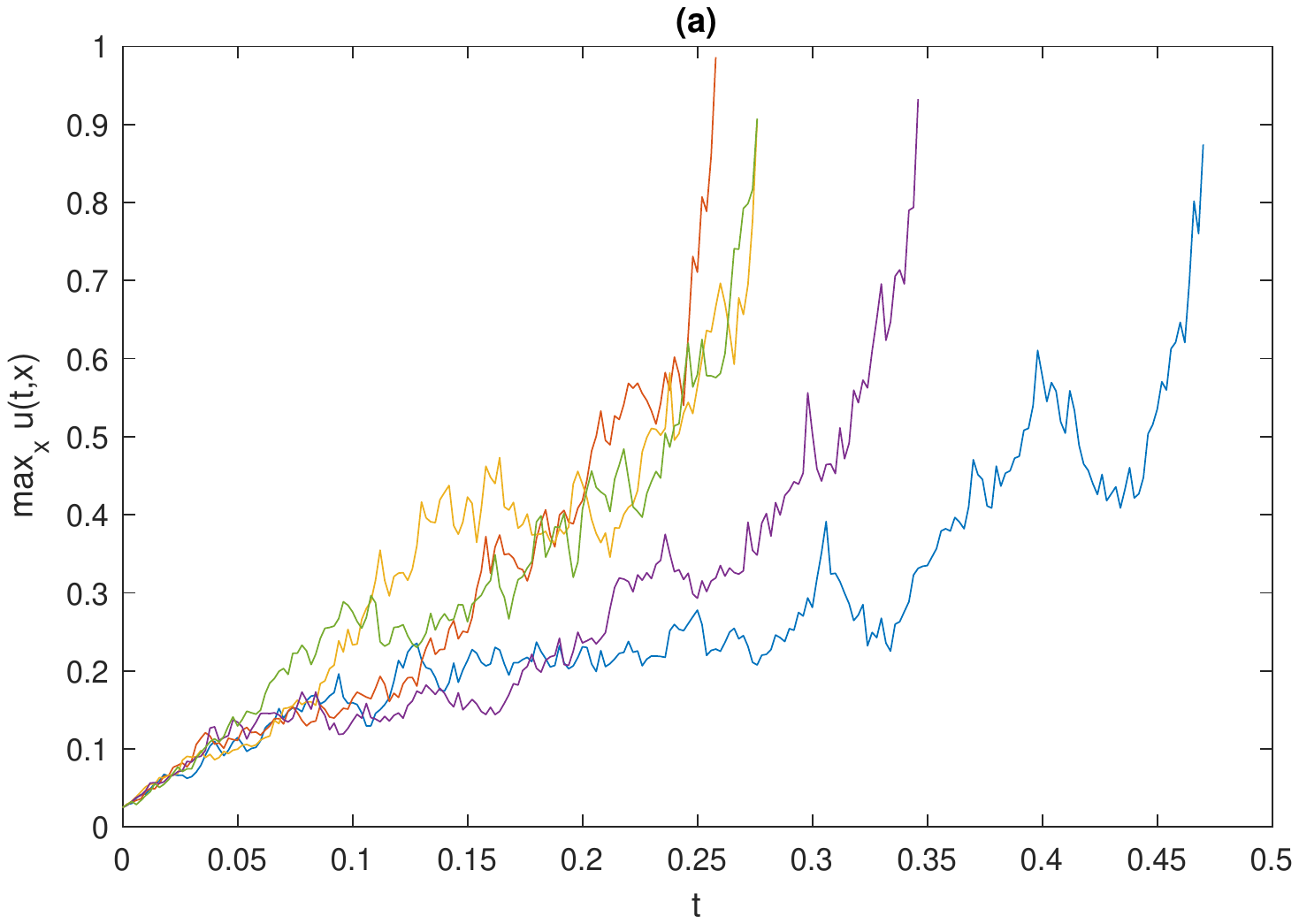}
   \end{minipage}\hfill
   \begin{minipage}{0.48\textwidth}\hspace{-2cm}
     \centering
     \includegraphics[width=1.2\linewidth]{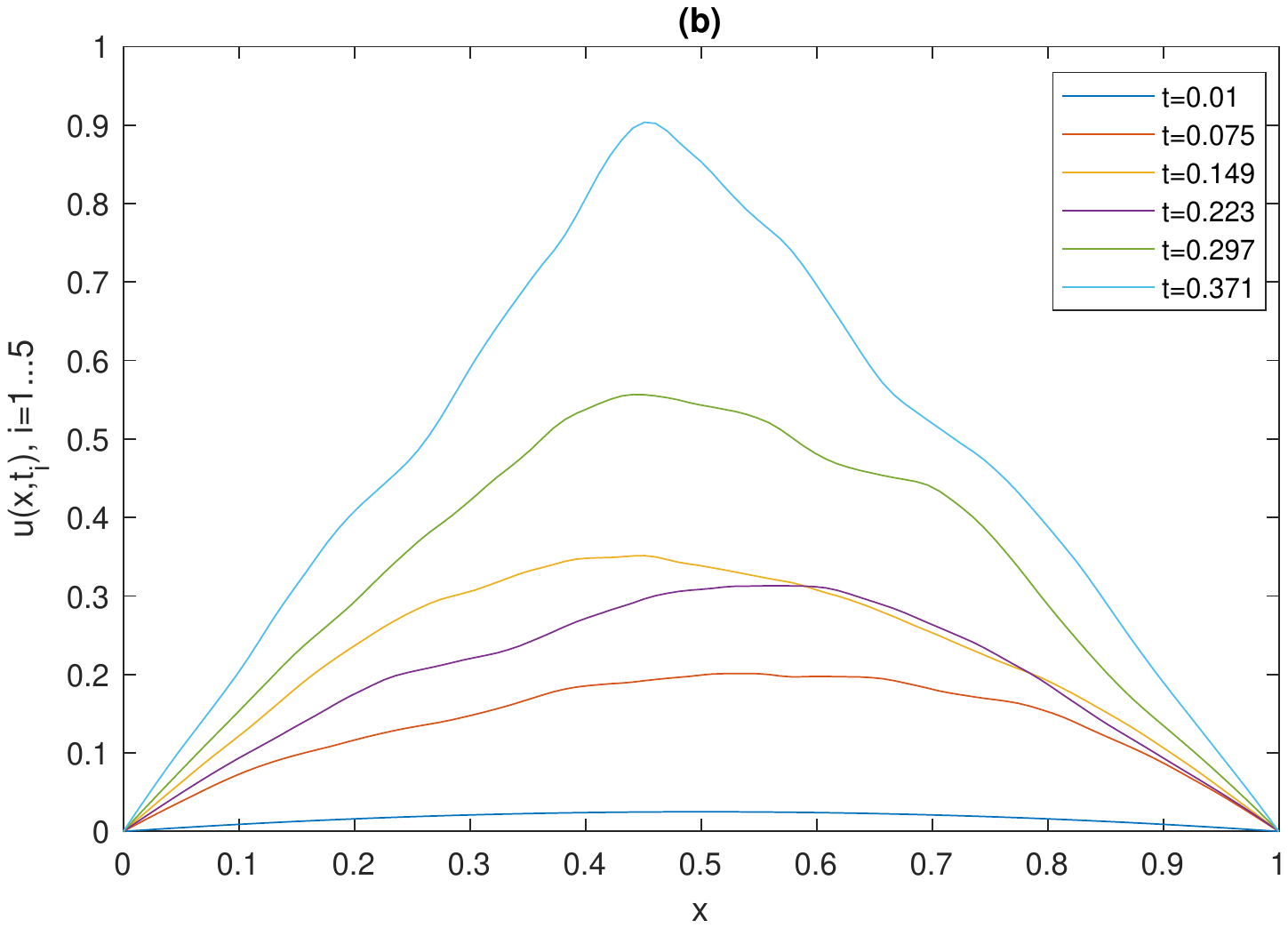}
   \end{minipage}\vspace{-4cm}
   \caption{(a) Realisation of the $\|u(\cdot,t) \|_\infty$ of the numerical solution of problem \eqref{LSP} for $\lambda=2$, $k=1$, $M=102$, $N=10e4$,  $r=0.1$  and initial condition $u(x,0)=c\,x(1-x)$ for $c=0.1$.
   (b) Plot of $u(x,t_i) $ from a different realization with the same  parameters values at five time instants.}\label{Fig_sim2}
\end{figure}

An interesting direction worth investigating is the derivation of  estimates  of the probability of quenching in a specific 
time interval $[0, T]$ for some $T>0$.
It is known, cf. \cite{Kavallaris2016},  that for  imposed Dirichlet boundary conditions, then the  solution $u$ will eventually quench in some finite time $T_q$ for large enough values of the parameter $\lambda$ or big enough initial data.

From the application  point of view an estimate of the probability that $T_q<T$ would be useful with respect to various values of the parameter $\lambda$.

   In Table (T1) the results of such a numerical experiment are presented. In particular, implementing  $N_R$ realizations then in the first column we print out  the values of $\lambda$  considered, 
  while the second column contains
     the number of times that the solution quenched before the time $T,$  whilst in the last two columns the mean $m(T_q)$ and the variance $Var (T_q)$ of the quenching time respectively are given.
    The rest of the parameters were taken to be the same as in the previous simulations but with $\kappa=0.1$.
  
  By the results in Table (T1) we observe that in a finite time interval the stochastic problem performs a dynamic behaviour which resembles  that of the  deterministic one. Specifically,  increasing the value of $\lambda$ initially we have no quenching in this time interval while after $\lambda>\lambda^*_{T}>1$ we have quenching almost surely at a time  $T_q$ with mean and variance decreasing with $\lambda$.  
 \bigskip
  \begin{center}\label{Table1}
{\bf Table (T1)}\\Realizations of the numerical solution of problem \eqref{LSP} \\
for $N_R=1000$ in the time interval $[0,10].$ \\
\bigskip
 \begin{tabular}{||c||c|c|c||}\hline\label{T1}
    $\lambda $      &        Quenching times      & $m(T_q)$  &  $\sigma^2 (T_q) $\\
      \hline
  0.5   & 0 & - &  -\\
 1 &   0 & - & - \\
1.5 &  1000 &  1.4642 & 0.0071 \\
2  &  1000 & 0.3542 &  3.7852e-05 \\
2.5   & 1000 & 0.2184 &  4.2468e-06\\
 \hline
\end{tabular}
\end{center}
Additionally,  we  perform  another experiment for simulation time $T=1$  and $\lambda=1.65,$  chosen in a $\la-$range  where the occurrence of quenching is not definite, 
 and for  a larger number of realizations $N_R=10^4$,  whilst 
  the rest of the parameter values being kept the same as in Table ($T_1$). Then
  we obtain a numerical estimation for the probability of quenching equal to $ 0.3464$ with $m(Tq)=0.3380$ and $Var(Tq)=0.2157.$

 Next  we consider the case of nonhomogeneous boundary conditions of the form \eqref{LSP2} or equivalently \eqref{LSP2za} with $\beta=\beta_c,$ since such a case is of  particular interest in the light of the quenching results of section \ref{eqp}.
A simulation implementing  the previously  described numerical algorithm for this particular case is presented in Figure \ref{Fig_sim1R}(a). The  presented  realization is for problem \eqref{LSP} and the  parameters used here are $\lambda=0.3$, $k=1$, $\beta=\beta_c=1.$  Also, in  Figure \ref{Fig_sim1R}(b) the quenching of $||u(\cdot,t)||_{\infty}$ for one realization is depicted.

\begin{figure}[!htb]\vspace{-4cm}\hspace{-1cm}
   \begin{minipage}{0.48\textwidth}
     \centering
     \includegraphics[width=1.2\linewidth]{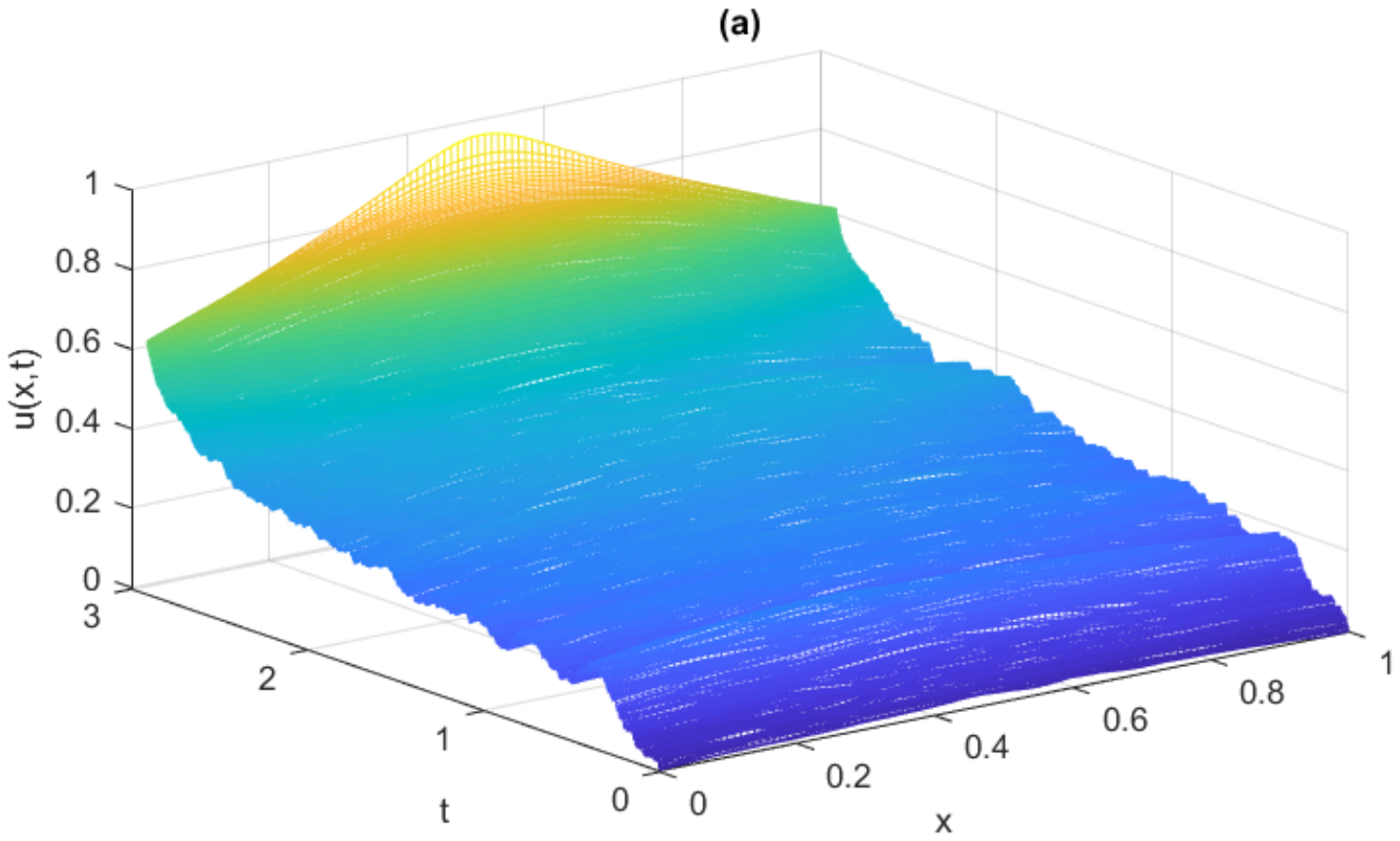}
   \end{minipage}\hfill
   \begin{minipage}{0.48\textwidth}\hspace{-2cm}
     \centering
     \includegraphics[width=1.2\linewidth]{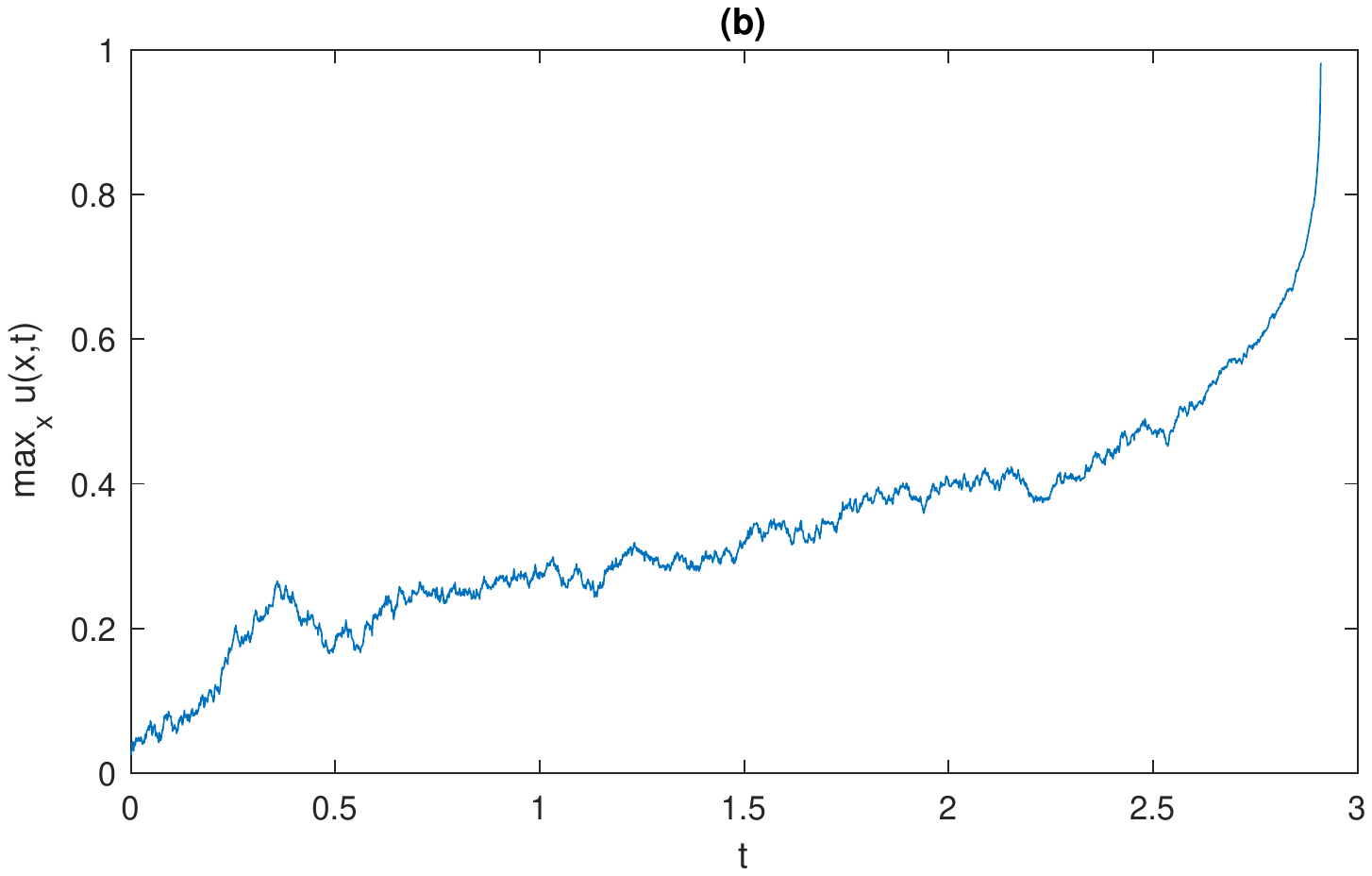}
   \end{minipage}\vspace{-4cm}
   \caption{(a) Realisation of the numerical solution of problem \eqref{LSP} for $\lambda=0.3$,  $\kappa=1$, $M=102$, $N=10e4$,  $r=0.1$,  initial condition $u(x,0)=c\,x(1-x)$ for $c=0.1$ and with $\beta=\beta_c=1$ in the nonhomogeneous boundary condition.
   (b) Plot of $\|u(\cdot,t) \|_\infty$. The quenching behaviour is apparent.}\label{Fig_sim1R}
\end{figure}
Similarly in the next set of graphs in Figure \ref{Fig_sim2R}(a) we display the quenching behaviour of five realisations of the numerical solution of problem \eqref{LSP} for $\lambda=0.3$. In an extra realization provided by  Figure \ref{Fig_sim2R}(b)  the spatial distribution of the numerical solution at different time instants is presented.
\begin{figure}[!htb]\vspace{-4cm}\hspace{-.5cm}
   \begin{minipage}{0.48\textwidth}
     \centering
     \includegraphics[width=1.2\linewidth]{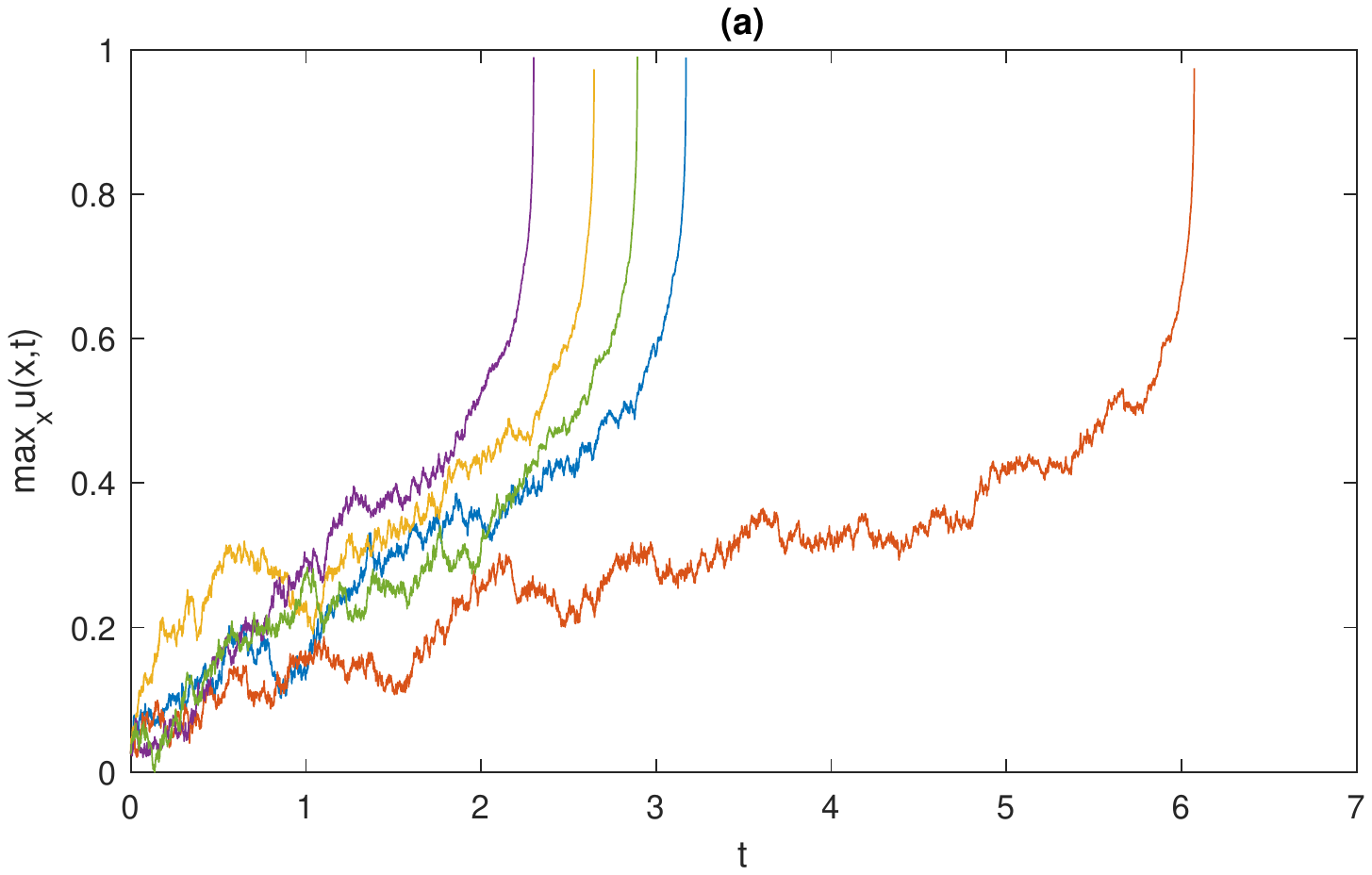}
   \end{minipage}\hfill
   \begin{minipage}{0.48\textwidth}\hspace{-2cm}
     \centering
     \includegraphics[width=1.2\linewidth]{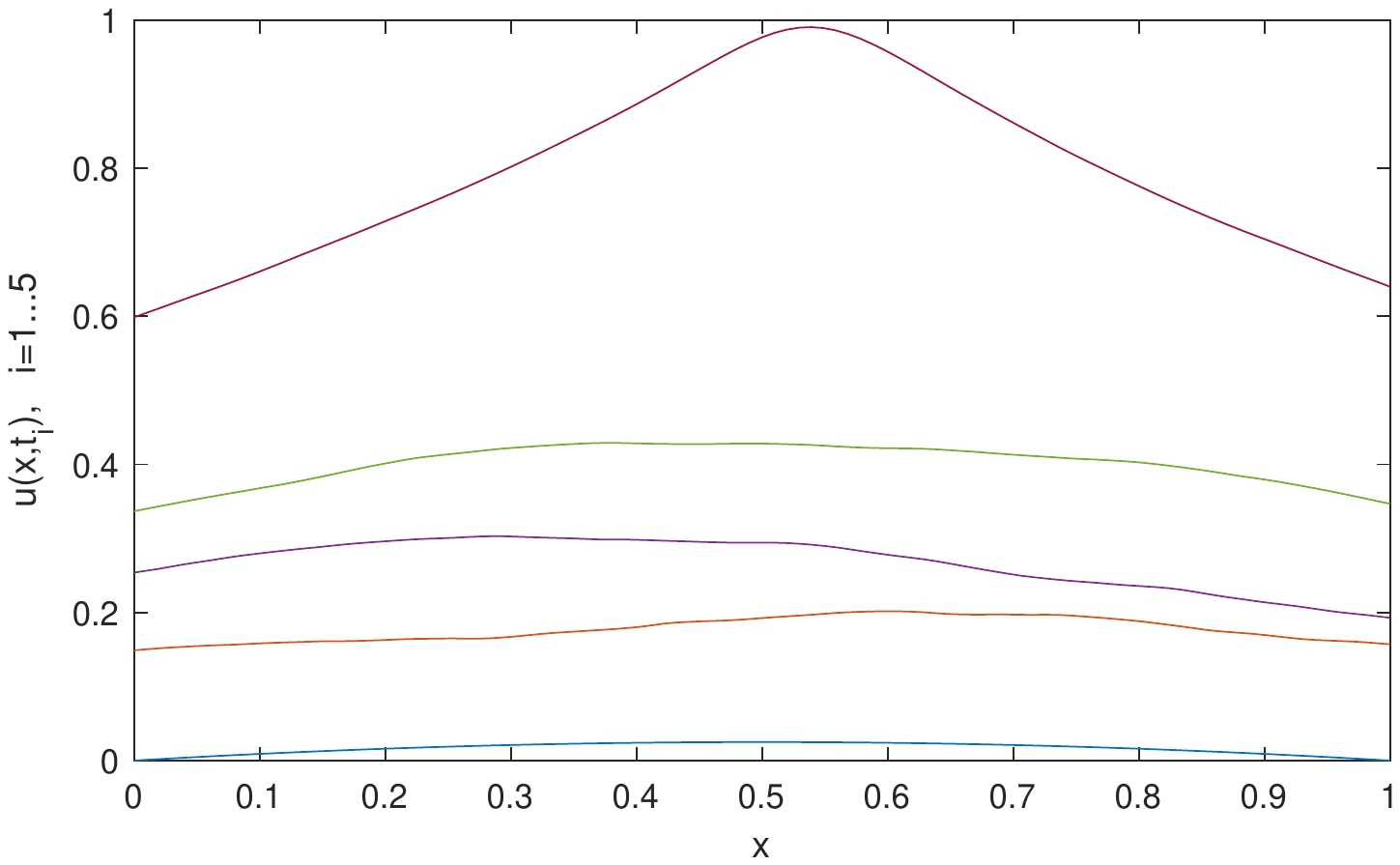}
   \end{minipage}\vspace{-4cm}
   \caption{(a) Realisation of the $\|u(\cdot,t) \|_\infty$ of the numerical solution of problem \eqref{LSP} for $\lambda=2$, $\kappa=1$, $M=102$, $N=10e4$,  $r=0.1$  and initial condition $u(x,0)=c\,x(1-x)$ for $c=0.1$.
   (b) Plot of $u(x,t_i) $ from a different realization with the same values of the parameters at five time instants.}\label{Fig_sim2R}
\end{figure}

Additionally in the following Table ($T_2$) we present the results of such a numerical experiment. Indeed, implementing  $N_R$ realizations we derive analogous results as  in Table ($T_1$).

\begin{center}\label{Table2}
{\bf  Table (T2)}\\Realizations of the numerical solution of problem \eqref{LSP} in the case of nonhomogeneous Robin boundary conditions  for $N_R=1000$ in the time interval $[0,1].$ \bigskip
 \begin{tabular}{||c||c|c|c||}
 \hline\label{T2}
    $\lambda $      &        Quenching times      & $m(T_q)$  &  $\sigma^2 (T_q) $\\
  \hline 
   0.2  & 0 & - & -  \\
 0.4 &  0 &       -  &  -  \\
0.6 & 0 & - & - \\
0.8  & 1000 & 0,75945 &	0.0014	\\
1 &   1000 &      0.5547	&   5.41553e-04\\ 
 \hline
\end{tabular}
%
\end{center}\vspace{-2cm}
We notice a transition of the behaviour of the solution $u$  around the value $\lambda\sim 0.7$. So, in the next table,  Table ($T_3$),  we focus  around this value and point out  a gradual increase of the number of quenching  results as the parameter $\lambda$ increases.

 \begin{center}\label{Table3}
{\bf  Table (T3)}\\Realizations of the numerical solution of problem \eqref{LSP} in the case of nonhomogeneous Robin boundary conditions \\ for $N_R=1000$ in the time interval $[0,1].$  \\
  \begin{tabular}{||c||c|c|c|| }\hline\label{T3}
    $\lambda $      &        Quenching times      & $m(T_q)$  &  $Var (T_q) $\\
  \hline
  0.6 & 0& -& -\\
  0.65  & 85 &      0.0841	&    0.0762\\
 0.675 &    594 &      0.5777	&    0.2284   \\
0.7 &      877&   0.8227	&    0.0958 	  \\
0.75  &   1000  &     0.8332	  &  0.0016	\\
 \hline
\end{tabular}
\end{center}

In the next set of experiments we solve numerically  problem \eqref{GSPzu}. We choose the diffusion coefficient to be of the form $g=g(t)=c_0+c_1\cos(\omega t)$, with $c_0=1$, $c_1=0.1$, $\omega=10$. We also consider a potential in the source term of the form  $h(x)=x^b$, for $b=\frac12$. The results of these experiments are demonstrated in Table ($T_4$). 
 \begin{center}\label{Table4}
{\bf  Table (T4)}\\Realizations of the numerical solution of problem \eqref{GSPzu} in the case of nonhomogeneous Robin boundary conditions  for $N_R=1000$ in the time interval $[0,1].$  
\\
  \begin{tabular}{||c||c|c|c||}\hline\label{T4}
    $\lambda $       &        Quenching times      & $m(T_q)$  &  $Var (T_q) $\\
  \hline
  0.6 & 0 & -& -\\
  0.8  &  0 & - & -  \\
 1 &    776 &   0.7348  &   0.1568   \\
1.2 & 1000 &  	 0.7432 	 &   0.0011  \\
1.4  &   1000  &   0.6029	 &   5.2540e-04	\\
 \hline
\end{tabular}
\end{center} 
Moreover focusing again around the value $\lambda\sim 1$ we can observe the transitional behaviour of the system in Table ($T_5$)  for $T=1.$
\begin{center}\label{Table5}
 {\bf  Table (T5)}\\Realizations of the numerical solution of problem \eqref{GSPzu} in the case of nonhomogeneous Robin boundary conditions  for $N_R=1000$ in the time interval $[0,1].$
 \\
  \begin{tabular}{||c||c|c|c|c||}\hline\label{T5}
    $\lambda $     &         Quenching times      & $m(T_q)$  &  $Var (T_q) $\\
  \hline
  0.9 & 0& 0& 0\\
  0.95  & 35  &     0.0347     &     0.0333\\
 0.97 &  191     &      0.1883          &       0.1505\\
0.99 &  431        &     0.4198         &    0.2331       \\
0.995  &  502       &     0.4870         &  0.2358      \\
1.1 &   993 &    0.8442 &  0.0068  \\
 \hline
\end{tabular}
\end{center}


\section{Discussion}\label{di}

In the current  work  we delmonstrate  an  investigation of  a $d-$dimensional,  $d=1,2,3,$  stochastic parabolic problem related to the modelling of an electrostatic
MEMS device part of which is a membrane-rigid plate system. Firstly,  the basic stochastic model  together is  presented. Later,  local existence and uniqueness of the basic  stochastic $u-$problem  \eqref{LSP}, as well as  of its main variations, and for general boundary conditions is established via Banach's Fixed point theorem. 

Next, and for a certain form of boundary conditions (cf. equation \eqref{LSP2za}) it is shown that the solution of $z-$problem   \eqref{LSPza} quenches   almost surely regardless the chosen initial condition or the value of the tuning parameter 
$\lambda.$ This is actually  a striking and counterintutive result; indeed  in almost every case quenching for the corresponding deterministic problem occurs only  if the parameter $\lambda$ or the initial data are  large enough. To the best of our knowledge, this the first result of such kind  is derived in the context of semilinear SPDEs related to MEMS.

Besides, adding  a regularizing term into equation \eqref{LSP1za}, in the form of a modified nonlinear drift term, changes the dynamics of solution  $z=1-u$ and we then obtain a dynamical  behaviour resembles that of the deterministic problem.  Moreover, in this particular case  a lower estimate of the quenching probability is provided by formula \eqref{LSPgProp}.

The case  of including time dependent coefficients related to dispersion and varying dielectric properties in the equation  is tackled by  similar analysis method. Again   a lower bound for the  quenching probability  or quenching almost surely are derived, depending on  the size of the first eigenvalue of corresponding eigenvalue problem.

We end our investigation by the implementation of an finite element numerical method, 
for the solution of the stochastic  time-dependent problem in the one-dimensional case. We also  provide a series of numerical experiments
initially  for the case of homogeneous  Dirichlet boundary conditions (for the $u$-problem)  and next for nonhomogeneous Robin conditions. In each case we present various results estimating the  quenching times  in a specific time interval $[0,T],$  which are of  particular interest for MEMS practitioners.



\end{document}